\documentclass[11pt]{article}

\makeatletter
\let\NAT@parse\undefined
\makeatother

\usepackage{amsmath,graphicx,psfrag,epsfig,xcolor,amsfonts,hyperref}
\usepackage{mathrsfs}  

\graphicspath{{./fig/}}
\usepackage{version}
\usepackage{subfigure}
\usepackage[square, sort&compress, numbers]{natbib}
 \usepackage[normalem]{ulem}
\usepackage{amsthm}
\usepackage{mathtools}
 \usepackage{enumitem}

\newcommand{\ZAA}[1]{{\color{black}{ #1}}}

\definecolor{mygreen}{RGB}{0, 100, 0}

\usepackage{bibentry}

\def\real{\mathbb{R}}

\newcommand\oprocendsymbol{\hbox{$\square$}}
\newcommand\oprocend{\relax\ifmmode\else\unskip\hfill\fi\oprocendsymbol}

\newtheorem{proposition}{Proposition}
\newtheorem{theorem}{Theorem}
\newtheorem{lemma}{Lemma}
\newtheorem{remark}{Remark}

\newtheorem{example}{Example}
\newtheorem{definition}{Definition}

\def \bs {\boldsymbol}

\def\B{K}
\def\C{C}

\def\H{H}

\def\nc{\mathcal{C}} 
\def\ino{\mathcal{K}} 
\newcommand\bit[1]{\textit{\textbf{#1}}}

\usepackage{caption}
\makeatother

\title{ Phase Reduction and Synchronization\\ of Coupled Noisy Oscillators}
\author{Zahra Aminzare  
\thanks{Zahra Aminzare is with the Department of Mathematics,  University of Iowa, IA, USA.
	{\tt\small zahra-aminzare@uiowa.edu}}
\and 
 Vaibhav Srivastava
 \thanks{Vaibhav Srivastava is with  	
	Electrical and Computer Engineering, Michigan State University, East Lansing, MI , USA.
	{\tt\small 	vaibhav@egr.msu.edu }}%
	}

\def \mc {\mathcal}
\begin{document}
\maketitle

\begin{abstract}

We study the synchronization behavior of a noisy network in which each system is  driven by two sources of state-dependent noise: (1) an intrinsic noise which is common among all systems and can be generated by the environment or any internal fluctuations, and (2) a coupling noise which is generated by interactions with other systems. After providing sufficient conditions that foster synchronization in  networks of general noisy systems, we focus on \textit{weakly} coupled networks of noisy \textit{oscillators} and, using the first- and second-order phase response curves (PRCs), we derive a reduced order stochastic differential equation to describe the corresponding phase evolutions. Finally, we derive  synchronization conditions based on the  PRCs and illustrate the theoretical results on a couple of models. 
\end{abstract}


{\section{Introduction}\label{sec:intro}

Coupled oscillator models are fundamental in modeling and analyzing the synchronization behavior of systems with rhythmic behavior, including systems in ecology, neuroscience, and engineering~\cite{winfree2001geometry, hoppensteadt2012weakly,demir2000phase,hajimiri1998general, goldbeter2002computational, Erment-Terman10, kuramoto2003chemical}. These simple models often miss environmental fluctuations as well as internal and external disturbances. Therefore, a stochastic dynamics approach provides a significant compromise to keep modeling complexity tractable and still capture important phenomena. 
 
Phase Response Curves (PRCs), which are computable both mathematically and experimentally  \cite{winfree2001geometry, 2007_Tateno_Robinson, 2010_Gouwens_etal, Schwemmer2012}, 
 provide fundamental information about how these oscillator models perform in a neighborhood of a stable limit cycle and facilitate a reduction of a high dimensional model to a 1-dimensional phase model. Furthermore, when multiple oscillator models interact with each other, such 1-dimensional reduced models enable the development of coupled oscillator models that use only the phase information and relative timing of their limit cycles. 

PRC theory is typically developed for small deterministic perturbations around a stable limit cycle and in such cases it is sufficient to consider only the first order effects of the perturbation on the limit cycle. In this paper, we consider stochastic perturbations to the limit cycle and develop a stochastic phase reduced model. 

The idea of phase reduction goes back at least to~\cite{Malkin49} and has been expanded and formalized in subsequent works,  including~\cite{winfree1974patterns, winfree2001geometry, guckenheimer1975isochrons}. The references~\cite{schultheiss2011phase, sacre2014sensitivity, brown2004phase} provide a good tutorial introduction to the topic.    
Phase reduction for noisy oscillators has also received remarkable attention 
\cite{hajimiri1998general,demir2000phase,teramae2009stochastic,2010_Schwabedal_Pikovsky, 
2013_Schwabedal_Pikovsky,2014_Newby_Schwemmer, Moehlis_2014, 2014_Thomas_Lindner,2017_Bonnin,bonnin2017amplitude,2018_Bressloff_MacLaurin,2019_Thomas_Lindner,2020_Bressloff_MacLaurin}. 
Compared with these works, we provide complementary techniques that illuminate phase reduction from a PRCs perspective. Ermentrout et al.~\cite{2011_Ermentrout_Beverlin_Troyer_Netoff} and 
 Teramae et al.~\cite{teramae2004robustness,teramae2009stochastic} consider a setup very similar to that studied in the present paper. However, their computations rely on the Stratonovich interpretation of stochastic differential equations, which leads to different reduced order models than those derived below using the It\^{o} interpretation.

Our main goal is to find conditions that foster synchronization in  networks of weakly coupled stochastically perturbed oscillators.  In such networks,  the oscillators can sense a common perturbation or  perturbation through their interactions with other oscillators in the network.  
Toward this end, we contribute two types of results. 
First, we provide conditions that foster synchronization in a network of systems which sense two different sources of noise: 
(1) an intrinsic state-dependent noise which is common among all systems and can be generated by the environment or any internal fluctuations, and (2) a state-dependent coupling noise which is generated by interactions with other systems. 
Although our goal is to study a network of \textit{weakly} coupled \textit{oscillators}, our first result is \textit{not} limited to such a network and is valid for more general networks. 
 Second, we develop a  stochastic phase reduced model for a  network of weakly coupled noisy oscillators, where we use the notion of first- and second-order phase response curves, and averaging theory for stochastic differential equations. 
 Finally, we apply the developed stochastic synchronization theory to the coupled phase equations to obtain the desired results. 

The remainder of the paper is organized as follows. In Section~\ref{sec:synchronization-two-coupled}, we prove the main results of this paper. After introducing noisy networks and defining stochastic synchronization, we provide conditions that foster stochastic synchronization in noisy networks. 
In Section~\ref{sec:background-phase-reduction}, we recall some background on {PRCs} and phase reduction. In Section~\ref{sec:single-phase-reduction}, we derive the phase reduced model for noisy oscillators and develop computational techniques to determine second-order {PRCs}. 
In Section~\ref{sec:coupled-phase-reduction}, we derive the phase reduced model for weakly coupled noisy oscillators
In Section~\ref{sec:synchronization-coupled-phase}, we apply the results of Section~\ref{sec:synchronization-two-coupled} to the phase reduced models in Section~\ref{sec:coupled-phase-reduction} to find conditions that foster synchronization in weakly coupled noisy oscillators. We illustrate these theoretical results on a couple of models. 
Finally, we conclude in Section~\ref{sec:conclusions}.  

\section{Stochastic synchronization  in noisy networks}\label{sec:synchronization-two-coupled}

In this section, we consider  a noisy network  of $N$ nonlinear systems with two sources of state-dependent noise: (1) an intrinsic noise which is common among all systems and can be generated by the environment or any internal fluctuations, and (2) a coupling noise which is generated by interactions with other systems. 
For $i=1,\ldots, N$, let the stochastic differential equation (SDE) 
 \begin{align}\label{eq:general-form2}
  d \phi_i \;=\;  \underbrace{\mathcal{F}(\phi_i\ZAA{,t})dt + \sigma{\ino}(\phi_i\ZAA{,t}) dW}_{\text{intrinsic dynamics}}
   \;+\underbrace{\displaystyle\sum_{j=1}^N c_{ij} \left(\epsilon\mathcal{H}(\ZAA{\phi_j,\phi_i}) dt +\delta \nc (\ZAA{\phi_j,\phi_i}) d  W_{ij}(t)\right)}_{\text{coupling dynamics}}
\end{align}
describe the dynamics of  system $i$ with state $\phi_i\in\mathbb{R}^n$. The intrinsic and coupling dynamics of system $i$ are described as below. 

\textbf{Intrinsic dynamics.} The systems are identical and governed by an $n-$dimensional vector of nonlinear functions, $\mathcal F$. 
There is a source of noise in  \eqref{eq:general-form2} which is common among all the systems in the network and described by $\sigma\ino(\phi_i\ZAA{,t}) dW$. The constant $\sigma\geq0$ is  the common noise intensity,  $\ino(\phi_i\ZAA{,t}) \in \mathbb{R}^{n\times n}$, and $W$ is an $n-$dimensional vector of independent standard Wiener processes. 

\textbf{Coupling dynamics.} Denote the underlying network graph by  $\mathcal{G}$ and assume that it is an undirected and  weighted graph with weight $c_{ij}$, i.e.,  $c_{ij}=c_{ji}\geq0$, with $c_{ij}>0$ if $i$ and $j$ are connected;  and $c_{ij}=0$ if $i$ and $j$ are not connected. 
The interaction between  system $i$ and another system, say $j$,  influences the dynamics of $i$ through  a deterministic term 
$c_{ij}\epsilon \mathcal{H}(\phi_j,\phi_i) dt$ and a stochastic term $c_{ij}\delta \nc(\phi_j,\phi_i) dW_{ij}$, where $\nc(\phi_j,\phi_i)\in\mathbb{R}^{n\times n}$, and 
$\mathbf{W}_i=(W_{i1}, \ldots, W_{iN})^\top$ is a vector of independent standard Wiener processes. 
The constants $\epsilon\geq0$ and $\delta\geq0$ respectively describe the  coupling strength and interaction noise intensity of the overall network while $\epsilon c_{ij}$ and $\delta c_{ij}$ respectively specify the coupling strength and  noise intensity of each connection. 
{We further assume that  if $\epsilon=0$ or $\mathcal H\equiv 0$ then $\delta=0$. }

For now, we only assume that $\mathcal F$, $\mathcal H$, $\ino$, and $\nc$ are nonlinear functions and they are nice enough so that  \eqref{eq:general-form2} has a unique solution, for example, they are Lipschitz and satisfy a linear growth condition. See \cite[Section 2.3]{Mao_book} for more details. Later in Theorems \ref{thm:cond:synch2} and \ref{thm:cond:synch3} below, we will discuss appropriate conditions of these functions. 

 Equation~\eqref{eq:general-form2} 
represents a broad range of network dynamics that can model many biological systems. For example, this framework covers the interconnected Kuramoto phase oscillators that model the brain's neural activity where the neural dynamics are subject to noise. The level of a functional connection between two regions is proportional to synchronization between the oscillators' phases associated with the two regions \cite{2019_Menara_Baggio_Bassett_Pasqualetti}.
As another example, the framework covers a coupled bursting models \cite{SIAM1,SIAM2} that approximate the dynamics of coupled central pattern generators (CPGs) \cite{Marder_Bucher_CPG_review, Ijspeert_CPG_review} which are complex networks of neurons that produce rhythmic behaviors, such as walking.  Synchronization properties and clusters formation of coupled CPGs explain the generation of various gait patterns in animal locomotions  \cite{ASH18,AH19}. 

After reviewing definitions of \textit{stochastic stability} and \textit{stochastic synchronization}, 
in Theorems \ref{thm:cond:synch2}-\ref{thm:cond:synch4}, we will  provide sufficient conditions that foster stochastic synchronization in \eqref{eq:general-form2}.

\begin{definition}[\bit{Stochastic stability}]\label{def:moment:stability}
Let $x(t)$ be a solution of an SDE. Then,
\begin{description}[leftmargin=*]
\item[Moment exponential  stability.]
$x(t)$ is $p-$th  ($p>0$) moment exponentially  stable if there are a pair of positive constants $C$ and $c$ and a neighborhood $\Omega_0$ of $x(0)$ such that for any  solution $y$ with $y(0)\in\Omega_0$ 
\begin{equation*}\label{eq:moment:stability}
\mathbb{E} \|y(t) - x(t)\|^p < C\;\mathbb{E} \|y(0) - x(0)\|^p e^{-c t},  \quad \forall t>0,
\end{equation*}
where $\mathbb{E}$ denotes the expected value and $\|\cdot\| $ denotes the Euclidean norm. When $p=2$, it is said to be exponentially stable in mean square. 

\item[Almost sure exponential stability.] $x(t)$ is almost sure exponentially  stable if there is a  neighborhood $\Omega_0$ of $x(0)$ such that for any  solution $y$ with $y(0)\in\Omega_0$ 
\begin{equation*}\label{eq:as:stability}
\limsup_{t\to\infty}\dfrac{1}{t}\;\log\|y(t) - x(t)\| <0, \quad \text{almost surely (a.s.), }
\end{equation*}
which means $\mathbb{P}\left\{ \limsup_{t\to\infty}\frac{1}{t}\;\log\|y(t) - x(t)\| <0\right\}=1$.
\end{description}
\end{definition}

\begin{definition}[\bit{Stochastic invariance}]\label{def:stochastic:invariance}
A set $\mathcal{S}$  is called an invariance set for an SDE, if for any $x_0\in\mathcal{S}$, 
\[\mathbb{P} \left\{ x(t)\in\mathcal{S}, \; \forall t\geq 0 \right\}=1,\]
where $x(t)$ is a solution of the SDE starting from $x_0$ at $t=0$. 

Moreover, for a fixed $p>0$, $\mathcal{S}$ is called $p-$th moment (respectively, almost sure) exponentially stable if any $x\in \mathcal{S}$ is $p-$th  moment  (respectively, almost sure) exponentially stable. 
\end{definition}

\begin{definition}[\bit{Stochastic synchronization}]\label{def:stochastic:synchronization}
Let $\mathcal{S}$ be the set of states $x=(x_1,\ldots,x_N)^\top$ such that $x_1=\cdots=x_N$, i.e., 
$\mathcal{S}:=\{x\;|\; x_1=\cdots=x_N \}$. 
We say that $\mathcal{S}$ is a synchronization manifold if it is stochastically invariant and ($p-$th moment or almost surely) exponentially stable. 
We say that a network stochastically synchronizes if it admits a synchronization manifold, i.e.,
there exist $c, C>0$ such that for any solution $x(t)$, there exists $s(t)\in \mathcal{S}$ such that
\begin{equation}\label{synch:moment:stability}
\mathbb{E} \|x(t) - s(t)\|^p < C\;\mathbb{E} \|x(0) - s(0)\|^p e^{-c t},  \quad \forall t>0,
\end{equation}
or for any solution $x(t)$, there exists $s(t)\in \mathcal{S}$ such that
\begin{equation*}\label{synch:as:stability}
\limsup_{t\to\infty}\dfrac{1}{t}\;\log\|x(t) - s(t)\| <0, \quad a.s.
\end{equation*}
\end{definition}

There have been some efforts to find conditions for synchronization in stochastic networks. For example, 
in \cite{Russo-Wirth-Shorten-2019}  
a sufficient condition for synchronization in a stochastic network of nonlinear systems is given. In this reference, the authors consider nonlinear state- (and time-) dependent diffusion matrices, however, they assume that the deterministic coupling is linear, i.e., $\mathcal{H}$ is assumed to be linear. In  \cite{Russo-Shorten-2018} 
the authors consider a stochastic network of nonlinear systems with linear coupling and a common noise. 
Here, we  study the synchronization properties of  a network of nonlinear systems in the presence of both  nonlinear noisy coupling  functions and nonlinear common noise. 
There are also some interesting results which guarantee synchronization onset in networks with no coupling but common noise, i.e.,  $\epsilon=\delta=0$ and 
$\sigma>0$, \cite{teramae2004robustness}.  

Although the systems in \eqref{eq:general-form2} can be of any arbitrary dimension, in the following theorem,  for the ease of notation, we assume that the state variables are 1-dimensional, $n=1$. 

We denote the Laplacian matrix of the underlying network graph $\mathcal{G}$ by $L_{[c]}$ (where the subscript $[c]$ represents the weights $c_{ij}$s) and its eigenvalues by $0=\lambda_{1,[c]} \leq \lambda_{2,[c]}\leq\cdots\leq \lambda_{N,[c]}$. 

\begin{theorem}[\bit{Stochastic synchronization: exponential stability in mean square}]\label{thm:cond:synch2}
  Fix $\Omega_1\subset \mathbb{R}$ and let $\Omega_2:=\{x-y \;|\: x, y \in \Omega_1\}$. 
Consider \eqref{eq:general-form2} and assume that:
\begin{enumerate}[leftmargin=*]
\item there exists a constant $ \bar c_{\mathcal{F}}$ such that for all $x,y\in \Omega_1$ \ZAA{and $t\geq0$},  
\[(x-y) (\mathcal{F} (x\ZAA{,t}) - \mathcal{F} (y\ZAA{,t})) \leq \bar c_{\mathcal{F}}(x-y)^2;\] 

\ZAA{\item $\mathcal{H}:\Omega_1\times\Omega_1\to\mathbb{R}$ satisfies
$\mathcal H(x,y)= -\mathcal H(y,x)$ and there exists a constant $ \underbar c_{\mathcal{H}}$  such that for all $x,y\in \Omega_1$, 
$  \underbar c_{\mathcal{H}}(x-y)^2 \leq (x-y) \mathcal{H}(x,y)$; 

\item there exists a non-negative constant $ \bar c_{\nc}$ such that for all $x,y\in\Omega_1$, $|\nc(x,y)| \leq \bar c_{\nc} |x-y|$;} and

\item there exists a  non-negative constant $ \bar{c}_{\ino}$ such that for all $x\in\Omega_1$ \ZAA{and $t\geq0$}, 
\[|\ino(x\ZAA{,t})-\ino(y\ZAA{,t})| \leq  \bar {c}_{\ino} |x-y|.\]
\end{enumerate}

Then for any solution $(\phi_1,\ldots,\phi_N)^\top$, 
\begin{equation*}\label{}
\mathbb{E} \sum_{i=1}^N|\phi_i(t) - \psi(t)|^2 < \mathbb{E}\sum_{i=1}^N|\phi_i(0) - \psi(0)|^2 e^{-c t},  \quad \forall t>0,
\end{equation*}
where $\psi(t)= \frac{1}{N}\sum_{i=1}^N \phi_i(t)$ and 
\begin{equation}\label{thm:c}
c: =-2\bar c_{\mathcal{F}} +  {2}\epsilon \underbar c_{\mathcal{H}}  \lambda-\delta^2  \bar c_{\nc}^2 (1-\frac{1}{N}) \lambda_{N,[c^2]} - \sigma^2 \bar c_{\ino}^2. 
\end{equation}
In \eqref{thm:c}, if $ \underbar c_{\mathcal{H}}>0$, then $\lambda=\lambda_{2,[c]}$, otherwise, $\lambda=\lambda_{N,[c]}$.  $\lambda_{N,[c^2]}$ denotes the largest eigenvalue of the Laplacian matrix of network graph $\mathcal{G}$ with weights $c_{ij}^2$. 

Therefore, the network stochastically synchronizes (in the sense of \eqref{synch:moment:stability} with $p=2$) when $c>0$. 
\end{theorem}

\begin{proof}
The proof has three main steps:
\begin{description}[leftmargin=*]
\item[Step 1. Introducing a synchronization manifold. ]
Let  $(\phi_1,\ldots,\phi_N)^\top$ be a solution of  \eqref{eq:general-form2},  $\psi(t):=\frac{1}{N}\sum_{i=1}^N \phi_i(t)$ be the average of $\phi_i$'s, and 
$e_i:=\phi_i-\psi$ be the corresponding error. 
The dynamics of $(e_1,\ldots, e_N, \psi)$ can be written as:
\small{ 
\begin{align}
  \left(\begin{array}{c}d e_1 \\\vdots \\ d e_N\end{array}\right)
&=
  \left(\begin{array}{cccc}
1-\frac{1}{N} & -\frac{1}{N}  & \cdots & -\frac{1}{N}  \\
 &  \ddots&  &  \\
-\frac{1}{N} & -\frac{1}{N} & \cdots & 1-\frac{1}{N}
\end{array}\right)_{N\times N}
\left(
 \left(\begin{array}{c}\mathcal{F}(e_1+\psi \ZAA{,t}) \\\vdots \\\mathcal{F}(e_N+\psi \ZAA{,t})\end{array}\right)+
   \epsilon\left(\begin{array}{c}{H_1}(e\ZAA{,\psi}) \\\vdots \\{H_N}(e\ZAA{,\psi})\end{array}\right)
    \right)dt\label{eq:error:e2}\\
    \nonumber \\
    &  +  \left(\begin{array}{cccc}
1-\frac{1}{N} & -\frac{1}{N}  & \cdots & -\frac{1}{N}  \\
 &  \ddots&  &  \\
-\frac{1}{N} & -\frac{1}{N} & \cdots & 1-\frac{1}{N} 
\end{array}\right)_{N\times N} 
\left(\begin{array}{cccc}
K(e,\psi \ZAA{,t}) \;|\; C_1(e\ZAA{,\psi}) \;|\;\cdots\;|\; C_N(e\ZAA{,\psi})\end{array}\right)
\left(\begin{array}{c}d{W} \\d\mathbf{W}_{1} \\\vdots \\d\mathbf{W}_{N}\end{array}\right),\nonumber\\
    \nonumber \\
    d\psi&=\dfrac{1}{N}\sum_{i=1}^N\left(\mathcal{F}(e_i+\psi \ZAA{,t})+\epsilon {H}_i(e\ZAA{,\psi})\right)dt
    +\dfrac{\delta}{N}\sum_{i,j=1}^N c_{ji}\nc(e_j,e_i)dW_{ij}
    + \dfrac{\sigma}{N} \sum_{i=1}^N \ino(e_i+\psi \ZAA{,t})dW, \label{eq:error:psi2}
\end{align}
}
where for $i=1,\ldots, N$, 
\vskip-0.4in
\begin{align*}
H_i(e\ZAA{,\psi}) &= \epsilon \sum_{j=1}^N c_{ij}\mathcal{H} (\ZAA{e_j+\psi,e_i+\psi}),\\ 
{K}(e,\psi \ZAA{,t})&=\sigma(\ino(e_1+\psi \ZAA{,t})-\ino(\psi \ZAA{,t}), \ldots, \ino(e_N+\psi \ZAA{,t})-\ino(\psi \ZAA{,t}))^\top,
\end{align*}
and $C_i(e\ZAA{,\psi})$ is an ${N\times N}$ matrix which its $i-$th row is $\delta(c_{i1}\nc(\ZAA{e_1+\psi,e_i+\psi}), \ldots, c_{iN}\nc(\ZAA{e_N+\psi,e_i+\psi})$ and its other rows are zero row vectors, and $d\mathbf{W}_{i}= (dW_{i1}, \ldots, dW_{iN})^\top$ is an $N-$dimensional Wiener increment.  
We denote the  $N\times N$ matrix in \eqref{eq:error:e2} by $A$. 

 Let   $e=(e_1,\ldots,e_N)^\top$ and $\bs y = (e_1,\ldots,e_N, \psi)^\top$, and define 
 $V(\ZAA{\bs y,t}) = \frac{1}{2}e^\top e$.   Note that  the set of zeros of 
 $V$ is  
 \[\mathcal{S}:=\{(e_1,\ldots,e_N,\psi\ZAA{,t})^\top \in \Omega_2^N\times \Omega_1\ZAA{\times[0,\infty)}\;|\; e_1=\cdots=e_N=0\}.\] 
  This set is a candidate for the desired synchronization manifold. In the following two steps we show that if $c>0$, 
  \ZAA{then $V$ becomes a Lyapunov function where $\dot V(\bs y,t)\leq -c<0$.} Then \ZAA{we conclude that} $\mathcal{S}$ is an exponentially stable invariant set for \eqref{eq:error:e2}-\eqref{eq:error:psi2}  \ZAA{and therefore} it is the synchronization manifold. 

\item[Step 2. Invariance of the synchronization manifold.] 
Note that the It\^{o} derivative of $V$ is equal to 
\[dV(\ZAA{\bs y,t}) = \mathcal {L}V(\ZAA{\bs y,t}) dt+ V_{\ZAA{\bs y}}\ZAA{(\bs y,t)}^\top g(\ZAA{\bs y,t}) dW,\]
where 
\begin{align}\label{lyap:operator2}
\mathcal{L} V(\ZAA{\bs y,t}) := \ZAA{V_t(\bs y,t)+}V_{\ZAA{\bs y}}(\ZAA{\bs y,t})^\top f(\ZAA{\bs y,t}) +\frac{1}{2} \mathrm{tr}\left[g^\top(\ZAA{\bs y,t}) V_{\bs y\bs y}(\ZAA{\bs y,t}) g(\ZAA{\bs y,t})\right]. 
\end{align}
\ZAA{The $(N+1)-$dimensional vectors} $f(\ZAA{\bs y,t})$ and $g(\ZAA{\bs y,t})$ are respectively the drift and diffusion terms of \eqref{eq:error:e2}-\eqref{eq:error:psi2},  \ZAA{$V_t=\frac{\partial V}{\partial t}=0$, $V_{\bs y}=\frac{\partial V}{\partial \bs y}= (e^\top,0)^\top$,
and $V_{\bs y\bs y}(\bs y,t)$ is the $(N+1\times N+1)$ Hessian matrix of $V$ which is a diagonal matrix with all entries equal to 1 except the last diagonal entry which is equal to 0. }
 The trace operator is denoted by $\mathrm{tr}[\cdot]$. We show that there exists $c_{\mathcal L}>0$ such that $\mathcal{L} V\leq -c_{\mathcal L} V$. Then by  \cite[Theorem 1]{Stanzhitskii_2001} we conclude that $\mathcal{S}$ is an invariant set for \eqref{eq:error:e2}-\eqref{eq:error:psi2}. 
\begin{itemize}[leftmargin=*]
\item Because $e_1+\cdots+e_N=0$,  $e^\top A = e^\top$, and $e^\top \left(\begin{array}{c}\mathcal{F}(\psi\ZAA{,t}) \\\vdots \\\mathcal{F}(\psi\ZAA{,t})\end{array}\right) = 0$. Therefore, the second term of the right hand side of \eqref{lyap:operator2} becomes:
\begin{align*}
V_{\bs y}(\bs y,t)^\top f(\ZAA{\bs y,t}) &= (e^\top,0)^\top f(e,\psi\ZAA{,t})&\text{} \\
&= (e_1,\ldots,e_N) \left\{\left(\begin{array}{c}\mathcal{F}(e_1+\psi\ZAA{,t}) \\\vdots \\\mathcal{F}(e_N+\psi\ZAA{,t})\end{array}\right)+
\left(\begin{array}{c}\mathcal{F}(\psi\ZAA{,t}) \\\vdots \\\mathcal{F}(\psi\ZAA{,t})\end{array}\right)
  +\epsilon\left(\begin{array}{c}{H_1}(e\ZAA{,\psi}) \\\vdots \\{H_N}(e\ZAA{,\psi})\end{array}\right)\right\} &\text{}\\
&= \sum_{i=1}^N e_i(\mathcal{F}(e_i+\psi\ZAA{,t}) -\mathcal{F}(\psi\ZAA{,t})) +\epsilon \sum_{i=1}^N  e_i H_i(e\ZAA{,\psi}) &\text{}
\end{align*}
By condition (i) and the definition of $V$,  the first sum satisfies
\begin{align*}
\sum_{i=1}^N e_i(\mathcal{F}(e_i+\psi\ZAA{,t}) -\mathcal{F}(\psi\ZAA{,t}))&\leq \bar c_{\mathcal{F}} \sum_{i=1}^N  e_i^2= 2\bar c_{\mathcal{F}}  V(\bs y,t).
\end{align*}
By  condition (ii) and using $c_{ij} = c_{ji}$, the second sum satisfies
\begin{align*}
\epsilon \sum_{i=1}^N  e_i H_i(e\ZAA{,\psi}) &= \epsilon \sum_{i=1}^N  e_i \sum_{j=1}^N c_{ji}
\mathcal{H}(\ZAA{e_j+\psi, e_i+\psi})&\text{}\\
 &= \frac{\epsilon}{2} \sum_{i=1}^N  \sum_{j=1}^N c_{ji} (e_i\mathcal{H}(\ZAA{e_j+\psi, e_i+\psi}) + e_j\mathcal{H}(\ZAA{e_i+\psi, e_j+\psi}))&\text{}\\
  &= - \frac{\epsilon}{2} \sum_{i=1}^N  \sum_{j=1}^N c_{ji} (e_i-e_j)\mathcal{H}(\ZAA{e_i+\psi, e_j+\psi})&\text{\ZAA{condition (ii)}}\\
    &< - \frac{\epsilon}{2} \sum_{i=1}^N  \sum_{j=1}^N c_{ji}  \underbar c_{\mathcal{H}}(e_i-e_j)^2&\text{condition (ii)}\\
        &= - {\epsilon}  \underbar c_{\mathcal{H}}e^\top L_{[c]} e, 
\end{align*}
Since $e^\top v_1=0$, where $v_1=(1,\ldots,1)^\top$ is the eigenvector of $L_{[c]}$ corresponding to $\lambda_{1,[c]} =0$, by min-max theorem,  $\lambda_{2,[c]} e^\top e \leq e^\top L_{[c]} e \leq \lambda_{N,[c]} e^\top e$.  Therefore, depending on the sign of $ \underbar c_{\mathcal{H}}$, we have:
\begin{align*}
 {\epsilon} \sum_{i=1}^N  e_i H_i(e\ZAA{,\psi}) &< -   {\epsilon} \underbar c_{\mathcal{H}}e^\top L_{[c]} e \leq - {\epsilon}  \underbar c_{\mathcal{H}}\lambda_{2,[c]} e^\top e =- {2}\epsilon  \underbar c_{\mathcal{H}}\lambda_{2,[c]} V(\ZAA{\bs y,t})&  \text{for $ \underbar c_{\mathcal{H}}>0$, or}\\
 {\epsilon} \sum_{i=1}^N  e_i H_i(e\ZAA{,\psi}) &<-  {\epsilon}  \underbar c_{\mathcal{H}}e^\top L_{[c]} e \leq - {\epsilon} \underbar c_{\mathcal{H}}\lambda_{N,[c]} e^\top e=- {2}\epsilon  \underbar c_{\mathcal{H}}\lambda_{N,[c]} V(\ZAA{\bs y,t}) &  \text{for $ \underbar c_{\mathcal{H}}<0$}. 
\end{align*}
Therefore, 
$
V_{\bs y}(\bs y,t)^\top f(\ZAA{\bs y,t}) \leq  (2\bar c_{\mathcal{F}} - {2}\epsilon  \underbar c_{\mathcal{H}}\lambda )V(\bs y,t). 
$
\item A straightforward  matrix multiplication implies that the third term of $\mathcal{L}V$ satisfies: 
\begin{equation*}
\frac{1}{2} \mathrm{tr}\left[g^\top(\ZAA{\bs y,t}) V_{\bs y\bs y}(\bs y,t) g(\ZAA{\bs y,t})\right] 
 =\frac{\delta^2}{2} \left(1-\frac{1}{N}\right) \sum_{i=1}^N  \sum_{j=1}^N c_{ij}^2  
 \nc^2(\ZAA{e_j+\psi, e_i,\psi}) + {\frac{1}{2}  \|A{K}(e,\psi\ZAA{,t})\|^2},
\end{equation*}
where  by condition (iii)
\begin{align*}
\frac{\delta^2}{2} \left(1-\frac{1}{N}\right) \sum_{i=1}^N  \sum_{j=1}^N c_{ij}^2  \nc^2(\ZAA{e_j+\psi, e_i,\psi})
& \leq\frac{ \delta^2}{2} \left(1-\frac{1}{N}\right)  \bar c^2_{\nc}\sum_{i=1}^N  \sum_{j=1}^N c_{ij}^2 (e_j-e_i)^2&\text{}\\
& =\frac{ \delta^2}{2} \left(1-\frac{1}{N}\right) \bar c^2_{\nc}e^\top L_{[c^2]} e&\text{}\\
& \leq\frac{ \delta^2}{2} \left(1-\frac{1}{N}\right)  \bar c^2_{\nc}\lambda_{N,[c^2]} e^\top e&\text{}\\
& =  \delta^2\left(1-\frac{1}{N}\right)   \bar c^2_{\nc}\lambda_{N,[c^2]} V(\bs y,t).&\text{}
\end{align*}
{Using the fact that the largest eigenvalue of $A$ is equal to one and} by condition (iv)
\begin{align*}
{\frac{1}{2}  \|A{K}(e,\psi\ZAA{,t})\|^2 \leq \frac{1}{2}  \|{K}(e,\psi \ZAA{,t})\|^2 = }
 \frac{\sigma^2}{2}  \sum_{i=1}^N (\ino(e_i+\psi\ZAA{,t}) - \ino(\psi\ZAA{,t}))^2
 \;\leq\; \frac{\sigma^2}{2}  \sum_{i=1}^N \bar c_{\ino}^2 e_i^2
 \;=\; \sigma^2 \bar c_{\ino}^2 V(\bs y,t). 
\end{align*}
\end{itemize}

Therefor, $\mathcal{L} V(\bs y,t)\leq -c_{\mathcal L} V(\bs y,t)$ where $c_{\mathcal L} =c=-2\bar c_{\mathcal{F}} + {2}\epsilon  \underbar c_{\mathcal{H}}  \lambda-\delta^2  \bar c_{\nc}^2 (1-\frac{1}{N}) \lambda_{N,[c^2]} - \sigma^2 \bar c_{\ino}^2.$ If $c>0$ then $\mathcal{L} V \leq -c V<0,$ and by  \cite[Theorem 1]{Stanzhitskii_2001}, 
 $\mathcal{S}$ becomes an invariant set for \eqref{eq:error:e2}-\eqref{eq:error:psi2}.

\item[Step 3. Stability of the synchronization manifold.]
As we discussed in Step 2, the It\^{o} derivative of $V$ is $dV(\bs y,t) = \mathcal {L}(V(\bs y,t)) dt+ V_{\bs y}^\top g(\bs y,t) dW$. 
By Dynkin's formula, $\mathbb{E} \int_s^t V_{\bs y}(\bs y(\tau), \tau)^\top g(\bs y(\tau),\tau) dW(\tau)=0$, and 
\begin{align*}
\mathbb{E} V(\bs y(t),t) - \mathbb{E} V(\bs y(s),s) &= \mathbb{E} \int_s^t dV(\bs y(\tau),\tau)\;d\tau 
&\text{}\\
&= \mathbb{E} \int_s^t \mathcal{L}V(\bs y(\tau), \tau)\;d\tau + \mathbb{E} \int_s^tV_{\bs y}(\bs y(\tau),\tau)^\top g(\bs y(\tau),\tau) dW(\tau) 
&\text{}\\
&= \mathbb{E} \int_s^t \mathcal{L}V(\bs y(\tau), \tau)\;d\tau
&\text{Dynkin's formula,}\\
&\leq -c\; \mathbb{E} \int_s^t V(\bs y(\tau),\tau) \;d\tau&\text{Step 2,}\\
&\leq -c\;  \int_s^t \mathbb{E} V(\bs y(\tau),\tau) \;d\tau&\text{Fubini's Theorem}. 
\end{align*}
The last inequality holds because $\mathbb{E} V(\bs y(\tau),\tau)$ is a continuous function of $\tau$ and hence its integral on $[s,t]$ is finite. 
Let $h(t)= \mathbb{E} V(\bs y(t),t)$, then by Gronwall's Inequality, $h(t)-h(s)\leq -c \int_s^t h(\tau)\; d\tau$ implies that $h(t)\leq h(0) e^{-ct}$. Hence, 
\begin{align*}
\mathbb{E} V(\bs y(t),t) \leq \mathbb{E} V(\bs y(0),0) e^{-ct} \quad\Rightarrow\quad  \mathbb{E} \|e(t)\|^2 \leq  \mathbb{E} \|e(0)\|^2 \; e^{-ct}, 
\end{align*}
or equivalently,
\begin{align*}\label{proof:exponential:stability:synch}
&\qquad\mathbb{E} \sum_{i=1}^N |\phi_i(t)-\psi(t)|^2 \leq  \mathbb{E} \sum_{i=1}^N |\phi_i(0)-\psi(0)|^2 \; e^{-ct}.&\text{}
\end{align*}
If $c>0$, then $\mathcal{S}$ becomes exponentially stable.  
Therefore, by the definition of stochastic synchronization and Step 2, $\mathcal{S}$ becomes a synchronization manifold. 
\end{description}
\end{proof}

In Theorem \ref{thm:cond:synch2}, we showed that if a network synchronizes in the absence of any noises (common noise or noise induced by the interactions among the nodes in the network), it could also synchronize in the presence of a sufficiently small noise and we found an upper bound for the noise intensities which guarantee such a behavior, i.e., we proved that if the noise intensities are such that $c>0$, then the network preserves its synchronization behavior. 
Unlike in Theorem \ref{thm:cond:synch2}, in the following theorem, $c$ can be negative and the system synchronizes. Indeed, the next result shows that noise intensities can be beneficial for networks synchronization.

\begin{theorem}[\bit{Stochastic  synchronization: $p-$th moment exponential stability}]\label{thm:cond:synch3}
Consider  conditions (i-iv) of Theorem \ref{thm:cond:synch2} and furthermore assume that 
\begin{enumerate}[leftmargin=*]

\item  \ZAA{$\nc:\Omega_1\times\Omega_1\to\mathbb{R}$ satisfies
$\nc(x,y)= -\nc(y,x)$ (or $\nc(x,y)= \nc(y,x)$) and there exists a constant $ \underbar c_{\nc}$  such that for all $x,y\in \Omega_1$, 
$\underbar{c}_{\nc} |x-y| \leq |\nc(x,y)|$;} and

\item there exists a  non-negative constant $\underbar{c}_{\ino}$ such that for all $x,y\in\Omega_1$ \ZAA{and $t\geq0$}, 
\[\underbar{c}_{\ino} (x-y)^2\leq (x-y) (\ino(x\ZAA{,t})-\ino(y\ZAA{,t})).\]
\end{enumerate}
 Let  
 \[\alpha_1=-c/2, \quad \alpha_2^2 = (\sigma\underbar{c}_{\ino})^2 + \frac{\delta^2\underbar{c}_{\nc}^2\lambda_{2,[c]}^2}{12},\]
 and assume that $0\leq  \alpha_1 < \alpha_2^2$. 
 Then for $0<p<2(1-\frac{\alpha_1}{\alpha_2^2})\leq2$ 
  and $\alpha:=-p[(\frac{p}{2}-1)\alpha_2^2-\alpha_1]$ (which is positive), 
 \eqref{eq:general-form2} stochastically synchronizes, i.e.,
 \begin{align*}
\mathbb{E} \left(\sum_{i=1}^N |\phi_i(t)-\psi(t)|^2\right)^{p/2}\leq  \mathbb{E} \left(\sum_{i=1}^N |\phi_i(0)-\psi(0)|^2\right)^{p/2} \; e^{-\alpha t}.&\text{}
\end{align*} 
\end{theorem}

To prove Theorem \ref{thm:cond:synch3}, we use the following lemma which is a modified version of \cite[Chapter 4, Corollary 4.6]{Mao_book}.

\begin{lemma}\label{lemma:cond:synch3}
Consider $dx= f(x\ZAA{,t}) dt +g(x\ZAA{,t})dW$ and assume that there exist constants  $\alpha_1$ and $\alpha_2$ such that \ZAA{for any $t\geq0$,}
\begin{align}
&x^\top f(x\ZAA{,t}) +\frac{1}{2} \mathrm{tr}[g^\top(x\ZAA{,t}) g(x\ZAA{,t})] \;\leq\; \alpha_1 x^\top x, \quad\text{and}\label{eq1:lemma:cond:synch3}\\
&\alpha_2 x^\top x \leq \|x^\top g\ZAA{(x,t)}\|\label{eq2:lemma:cond:synch3}.  
\end{align}
If $0\leq \alpha_1<\alpha_2^2$, then the trivial solution of $dx= f(x\ZAA{,t}) dt +g(x\ZAA{,t})dW$ is $p-$th moment exponentially stable provided $0<p<2(1-\frac{\alpha_1}{\alpha_2^2})\leq2$ 
  and $\alpha:=-p[(\frac{p}{2}-1)\alpha_2^2-\alpha_1]>0$, i.e., $\forall t>0$
\begin{equation*}
\mathbb{E} \|x(t)\|^p < \mathbb{E} \|x(0)\|^p e^{-\alpha t}.
\end{equation*}
\end{lemma} 

\bit{Proof of Theorem \ref{thm:cond:synch3}. } 
Under the conditions of Theorems \ref{thm:cond:synch2} and \ref{thm:cond:synch3}, we apply Lemma \ref{lemma:cond:synch3} to \eqref{eq:error:e2}. The left hand side of \eqref{eq1:lemma:cond:synch3} is equivalent to $\mathcal L(V(\bs y,t))$ which we showed $\mathcal L(V(\bs y,t))\leq -c V(\bs y,t)=-\frac{c}{2} e^\top e$. Therefore, $\alpha_1= -c/2$. 
A straightforward matrix multiplications gives:
\small{
\begin{align*}
 \|e^\top g\|^2 &=\|e^\top  A\left(\begin{array}{cccc}K(e,\psi,t) \;|\; C_1(e,\psi) \;|\;\cdots\;|\; C_N(e,\psi)\end{array}\right)\|^2&\text{}\\ 
 &=\|e^\top \left(\begin{array}{cccc}K(e,\psi,t) \;|\; C_1(e,\psi) \;|\;\cdots\;|\; C_N(e,\psi)\end{array}\right)\|^2&\text{$e^\top A=e^\top$}\\ 
   &=\sigma^2 \left(\sum_{i=1}^N e_i\ino(e_i+\psi,t)\right)^2
    + \delta^2 \sum_{i=1}^Ne_i^2 \sum_{j=1}^N c_{ij}^2 \nc^2(e_j+\psi,e_i+\psi)&\text{}\\
  &=\sigma^2\left(\sum_{i=1}^N (e_i+\psi - \psi)(\ino(e_i+\psi,t)-\ino(\psi,t))\right)^2 &\text{$\ino(\psi)\sum e_i=0$}\\
  &\quad+ \frac{\delta^2}{2}\sum_{i,j=1}^N c_{ij}^2(e_i^2+e_j^2) \nc^2(e_i+\psi,e_j+\psi)&\text{$\nc(x,y)=\pm\nc(y,x)$ }\\
  &\geq (\sigma\underbar{c}_{\ino})^2\left(\sum_{i=1}^N e_i^2\right)^2&\text{conditions (ii) \& (i) of Theorem \ref{thm:cond:synch3}}\\
 &\quad+ \frac{\delta^2}{2} \frac{\underbar c_{\nc}^2}{2} \sum_{i,j=1}^N c_{ij}^2(e_i-e_j)^4&\text{and $e_i^2+e_j^2 \geq \frac{(e_i-e_j)^2}{2}$} \\ 
 &\geq\left((\sigma\underbar{c}_{\ino})^2 + \frac{\delta^2}{2} \frac{\underbar c_{\nc}^2}{2}\frac{ \lambda_{2,[c]}^2}{{N}}\right)(e^\top e)^2.
  \end{align*}
  }
The last inequality holds because
$\sum c_{ij}^2(e_i-e_j)^4 = \sum (c_{ij}(e_i-e_j)^2)^2\geq \frac{1}{{N}} \left(\sum c_{ij}(e_i-e_j)^2\right)^2.$
Therefore, $\alpha_2^2 = (\sigma\underbar{c}_{\ino})^2 + \dfrac{\delta^2\underbar{c}_{\nc}^2\lambda_{2,[c]}^2}{{4N}}$. By Lemma \ref{lemma:cond:synch3}, for $p<2-2{\alpha_1}/{\alpha_2^2}$, 
 \begin{align*}
\mathbb{E} \left(\sum_{i=1}^N |\phi_i(t)-\psi(t)|^2\right)^{p/2}\leq  \mathbb{E} \left(\sum_{i=1}^N |\phi_i(0)-\psi(0)|^2\right)^{p/2} \; e^{-\alpha t},&\text{}
\end{align*}
where  $\alpha=p\alpha_1 -p(p/2-1)\alpha_2^2$.
\qed

\begin{theorem}[\bit{Stochastic  synchronization: almost sure exponential stability}]\label{thm:cond:synch4}
Consider conditions (i-iv) of Theorem \ref{thm:cond:synch2} and conditions (i-ii) of Theorem \ref{thm:cond:synch3}. Then 
\[\limsup_{t\to\infty}\dfrac{1}{t}\;\log(\|e(t)\|) \leq \alpha_1-\alpha_2^2, \quad a.s.\]

\end{theorem}
\begin{proof}
Proof by \cite[Chapter 4, Theorem 4.3]{Mao_book}. 
\end{proof} 

\begin{remark}Consider a deterministic network which does not synchronize, i.e., $2\bar c_{\mathcal{F}} - {2}\epsilon  \underbar c_{\mathcal{H}}  \lambda\geq0$.  Theorem \ref{thm:cond:synch4} guarantees  that for $\bar {c}_{\ino}^2<2\underbar{c}_{\ino}^2$,  a common noise, with sufficiently large  intensity, forces  the network  to synchronize.
\end{remark}

{
\begin{example}
Consider the following variation of the leader-follower oscillator.
\[
d \phi_i = \sin(r\phi_i - \phi_i^3)dt + \sigma \sin(\phi_i) dW +  \sum_{j=1}^N c_{ij} \left(\epsilon\sin(\phi_j -\phi_i) dt +  \delta \sin(\phi_j -\phi_i)  dW_{ij}\right),
\]
where {$\epsilon$} is the coupling strength, {$c_{ij}$}  is unity if oscillator $i$ and $j$ are neighbors in the interaction graph, otherwise it is zero, $\sigma\ge 0$ is the common noise strength,  {$\delta>0$}  is the coupling noise intensity, and $sin(\phi_i) dW$ is the added common noise.  Consider three particles interacting on a line graph and select parameters $\epsilon=1.3$, $r=5$, and $\delta=0.1$. Figure~\ref{fig:example1} shows the evolution of state and the norm of disagreement vector in the state for $\sigma=\{0,1\}$. Disagreement is computed by $\bs e = L \bs \phi$, where $L$ is the Laplacian matrix and $\bs \phi$ is the vector of $\phi_i$. 

\begin{figure}
	\subfigure[$\sigma=0$]{
\includegraphics[width=0.23\textwidth]{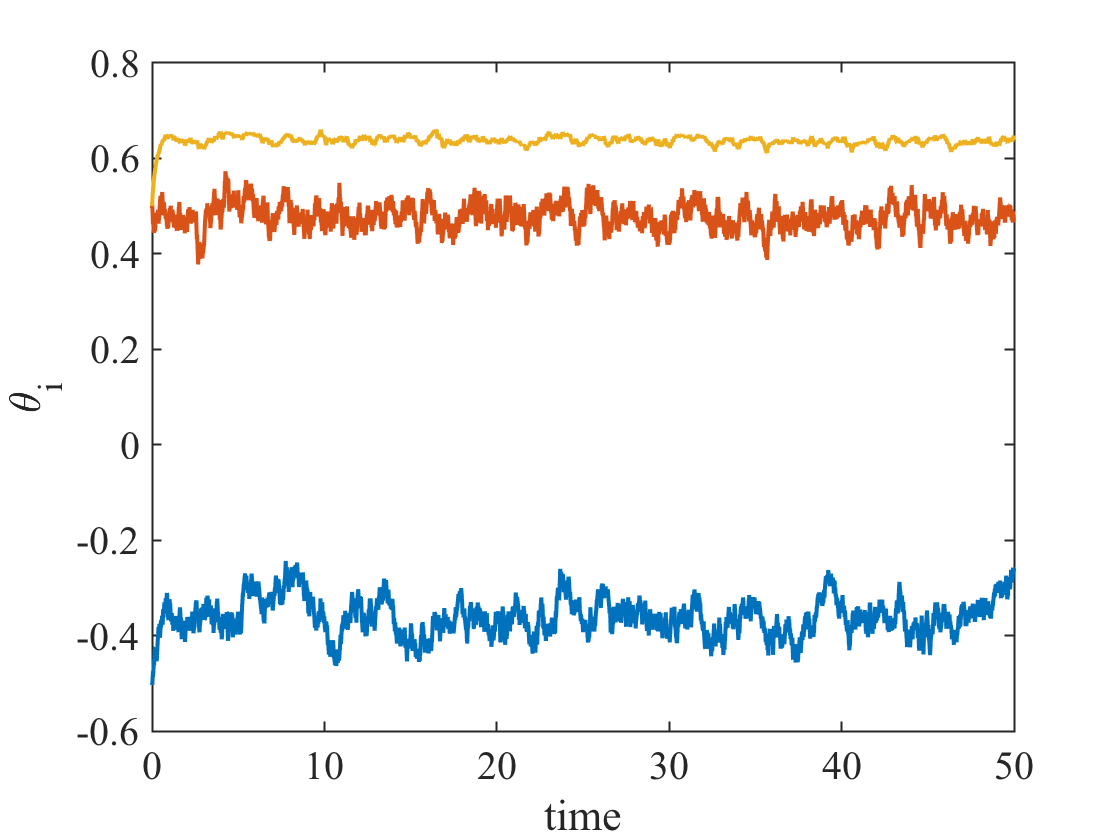}	
\includegraphics[width=0.23\textwidth]{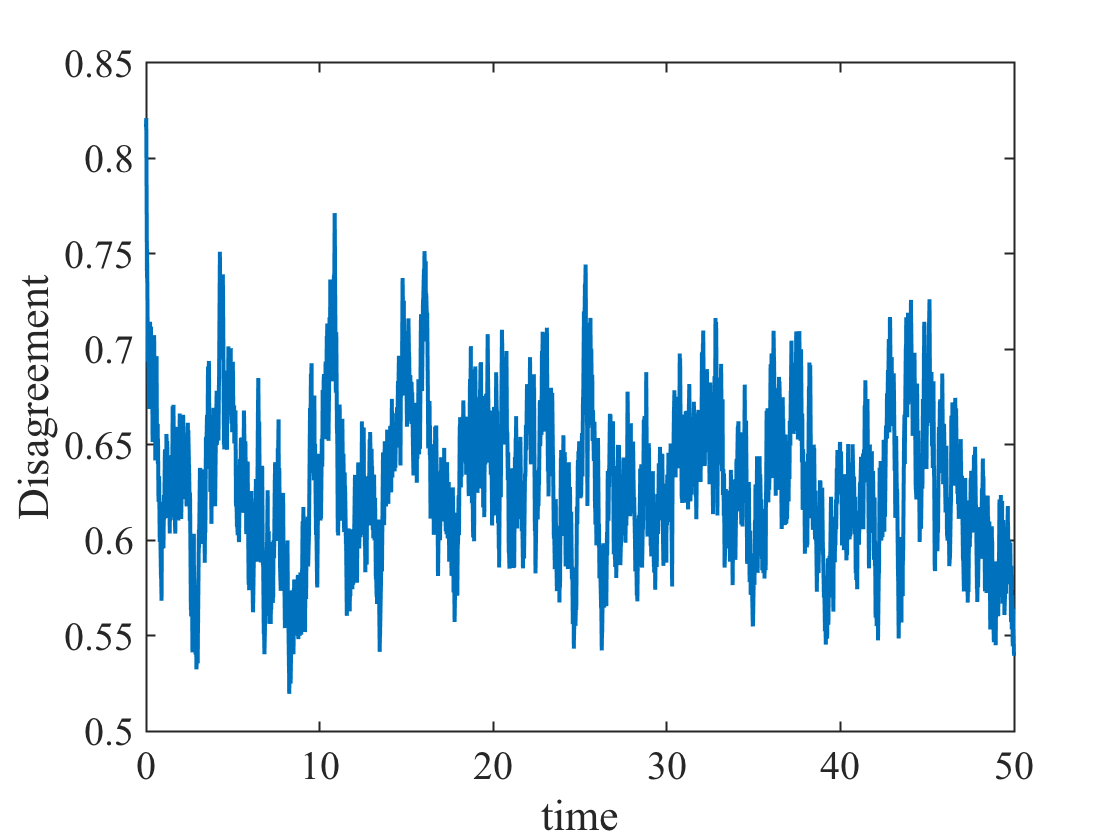}	
}
	\subfigure[$\sigma=1$]{
	\includegraphics[width=0.23\textwidth]{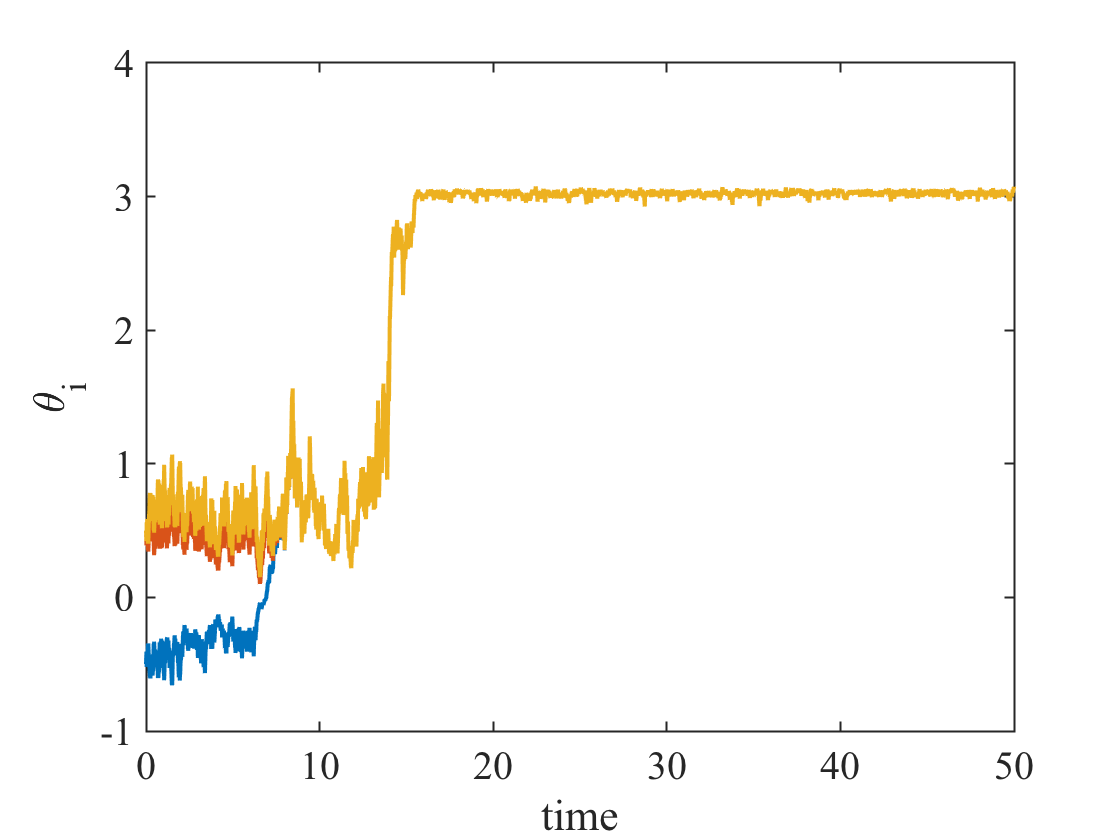}	
	\includegraphics[width=0.23\textwidth]{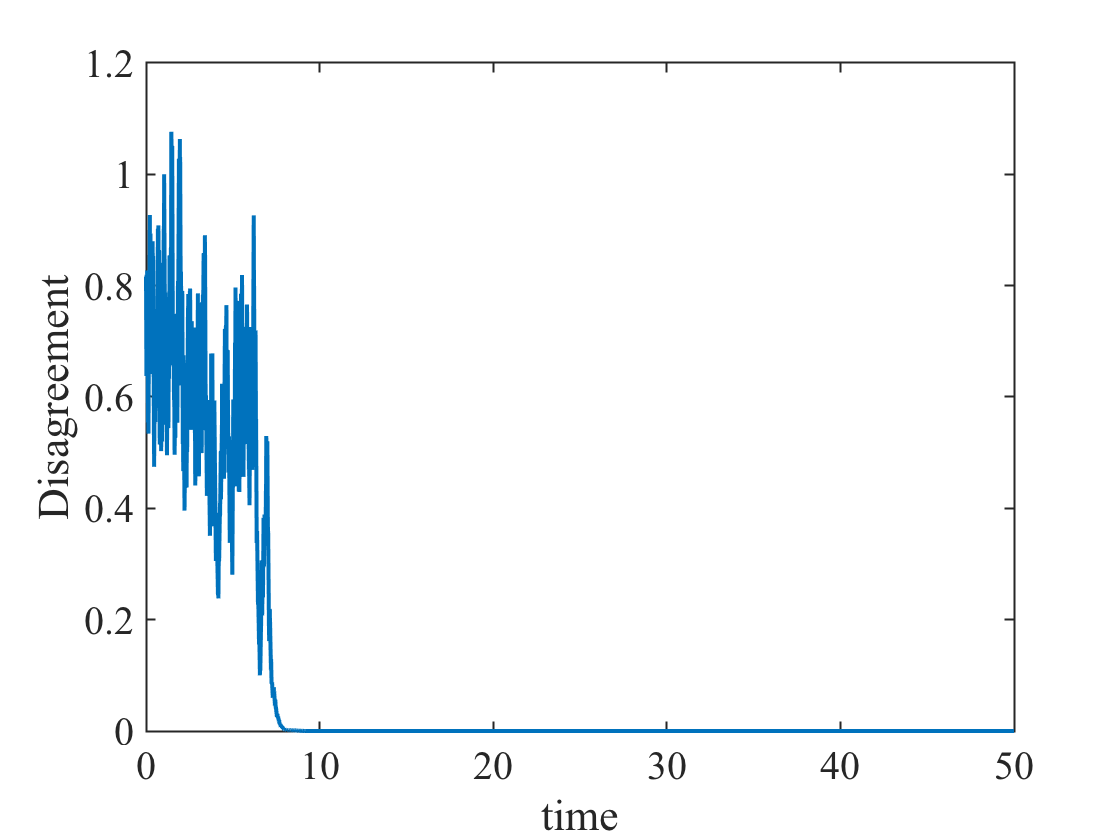}	
}
\caption{(a) {In the absence of a common noise, the oscillators do not synchronize. (b) Addition of common noise enables synchronization of oscillators. \label{fig:example1}}}
\end{figure}
\end{example}
}

\section{Phase reduction: A short review}\label{sec:background-phase-reduction}

\subsection{Phase reduction for deterministic oscillators} \label{subsec:phase-red}

In this section we first review the phase reduction technique for oscillators. Consider the autonomous system
\begin{equation}\label{eq:oscillator-model}
\dot{x} = F (x), \quad x \in \real^m, \quad m \ge 2,
\end{equation}
with an asymptotically stable hyperbolic limit cycle $x^\gamma(t)$ of period $T$  and frequency $\omega= \frac{2\pi}{T}$. 
The {\em phase} of an oscillator, denoted by $\theta(t)$, is the time that has elapsed as its state
moves around the limit cycle $x^\gamma(t)$, starting from an arbitrary reference point $\bar\theta$
on the cycle, called relative phase.  The phase, defined by 
$\theta(t) = \omega t +\bar\theta\quad (\mbox{mod}\; 2\pi)$, reduces the dynamics of  (\ref{eq:oscillator-model}) to the following scalar phase equation
\begin{equation*}\label{eq:phase}
\dot\theta(t) = \omega.
\end{equation*}
There is a one to one correspondence between phase $\theta$ and each point $x$ on the limit cycle $x^\gamma(t)$. This correspondence defines the following {\em phase map}~\cite{guckenheimer1975isochrons,schultheiss2011phase} on the basin of attraction {of} $x^\gamma(t)$:
\begin{equation*}\label{eq:phase_map}
\Phi(x(t)) := \theta(t) = \omega t+\bar \theta, \quad x\in x^\gamma,
\end{equation*}
with dynamics 
\begin{equation}\label{eq:phase_mapODE}
\dot\Phi(x(t)) =
\nabla\Phi(x)\cdot \dot{x} =
\nabla\Phi(x) \cdot  F(x) =  \omega, \quad x\in x^\gamma, 
\end{equation}
{where $\cdot$ denotes the inner product.}

We now {consider the effect of small perturbations to the dynamics of (\ref{eq:oscillator-model}), which no longer leave $x^\gamma(t)$ invariant.}
To this end, we first generalize the definition of the phase map to a neighborhood of $x^\gamma(t)$. Since $x^\gamma(t)$ is asymptotically stable, for any point $y$ in the basin of attraction of $x^\gamma(t)$, there exists an $x\in x^\gamma(t)$ such that as $t\to\infty$, 
$\|X(t,x) - X(t, y)\| \to 0$, where $X(\cdot, x)$ is the unique solution of (\ref{eq:oscillator-model}) with initial condition $x$, and $\|\cdot\|$ is an arbitrary norm in $\mathbb{R}^n$. The set of all such points $y$ is called the \emph{isochron} of $x$. 
For any $x\in x^\gamma$, all the points on the isochron of $x$ have the same phase as  $x$, i.e., $\Phi$ can be extended to the basin of attraction of $x^\gamma$ as follows:
\[
\Phi(y) :=\Phi(x) =  \theta, \quad \forall  y \in \mbox{isochron of $x$}.  
\]
Note that the isochron of $x$ is a level set of $\Phi(x)$. 

Now consider (\ref{eq:oscillator-model}) with a small perturbation $G$, which could describe the coupling to other oscillators. 
\begin{equation}\label{eq:oscillator_perturbed}
\dot{x} = F (x) +\epsilon G(x, \ldots), \quad x \in \real^m, \quad 0< \epsilon\ll 1. 
\end{equation}
Then, using \eqref{eq:phase_mapODE}, we have
\begin{equation*}\label{eq:Phase_Reduction}
\dot\Phi(x(t)) =
\nabla\Phi \cdot \dot{x} =
\nabla\Phi \cdot (F+\epsilon G) =
\omega + \epsilon\; \nabla\Phi\cdot G. 
\end{equation*}
Therefore, by definition of the phase map, the dynamics of (\ref{eq:oscillator_perturbed}) can be reduced to the following phase equation 
\begin{equation*}\label{}
\dot\theta=
\omega + \epsilon\;\nabla\Phi\cdot  G . 
\end{equation*}
The gradient of the phase map, $\nabla\Phi$, called the phase response curve (PRC) and captures changes in the phase per unit perturbation for small perturbations, plays an important role in reducing the system \eqref{eq:oscillator_perturbed}. 
In this work, we also use the Hessian of the phase map. To distinguish between the gradient and the Hessian, we refer to the gradient as the \textit{first-order PRC} and the Hessian as the \textit{second-order PRC}. 
In what follows, we review two essential methods to compute first-order PRCs. 

\subsection{Computation of  first-order phase response curves} \label{subsec:prc}

The  first-order PRC, denoted by $\bs Z(\theta)$, is defined by the gradient of the phase map $\nabla\Phi(x)$ at the point on the limit cycle associated with phase $\theta$. The first-order PRC can be computed using a \emph{direct method} in which a perturbation is introduced at each point of the limit cycle and the resulting change in phase is recorded,
\begin{equation*}\label{eq:PRC-direct}
\bs Z_i(\theta)= \frac{\partial  \Phi}{\partial x_i} ({x}) = \lim_{r \to 0} \frac{1}{r}\left(\Phi({x} + r \hat{ i}) -\Phi(x)\right), 
\quad  x \in x^\gamma,
\end{equation*}
where $\hat {i}$ is the $i$-th coordinate vector. 

Alternatively, the \emph{adjoint method} can be used that solves the following ODE in reverse time \cite{Erment-Terman10, Schwemmer2012}
\begin{equation}\label{eq:adjoint-method}
\frac{d}{dt}\nabla {\Phi} (x^\gamma(t))  =- D  F^\top (x^\gamma(t))\nabla {\Phi} (x^\gamma(t)),
\end{equation}
with constraint 
\begin{equation}\label{eq:adjoint-method-constraint}
\nabla {\Phi} (x^\gamma(0))\cdot F(x^\gamma(0)) = \omega, 
\end{equation}
where $D F^\top$ denotes the transpose of the Jacobian of $F$. 
Note that due to the negative sign in front of $-DF^\top$, the stability of the adjoint equation~\eqref{eq:adjoint-method} is the opposite of the stability of the limit cycle. Hence, the adjoint equation needs to be solved in reverse time. 
   
\section{Phase reduction for a single noisy oscillator}\label{sec:single-phase-reduction}

We now focus on the effect of noise in the dynamics of equation~\eqref{eq:oscillator-model} {and} its phase reduction. 
Consider~\eqref{eq:oscillator-model} with multiplicative 
white noise
\begin{equation}\label{eq:noisy-oscillator}
dx =  F(x) dt + \sigma K(x) dW(t),
\end{equation}
where $K(x) \in \real^{m\times m}$ is the
 {diffusion} matrix, $0<\sigma \ll 1$ is a constant determining the intensity of noise, and $d W(t)$ is a standard $m$-dimensional Weiner process increment. 
We could equivalently write the corresponding Langevin system 
\begin{equation*}\label{eq:oscillator_coupled_noisy_Langevin}
\dfrac{d{x}}{dt} = F (x) + \sigma K(x)   \xi(t),
\end{equation*}
where $\xi$ is $m-$dimensional white Gaussian noise. Therefore, we interpret  \eqref{eq:noisy-oscillator}  in It\^{o} sense. 

In the absence of  noise, $\sigma=0$, the  SDE \eqref{eq:noisy-oscillator} becomes an ODE $\dot{ x} =  F(x)$,  that we assume it admits  an asymptotically stable limit cycle $x^\gamma(t)$ with time period $T$ and frequency $\omega= \frac{2\pi}{T}$.  
We denote $x^\gamma(t)$ by $\gamma(\theta(t)) = \gamma(\omega t+\psi(t))$, where $\theta (t)$ and $\psi(t)$ are the asymptotic phase and relative phase, respectively. 
By a {\it noisy oscillator}, denoted by $x(t)$, we mean the solution of~\eqref{eq:noisy-oscillator} which can be approximated by $x^\gamma(t)$ for sufficiently small $\sigma$. 

Before we discuss the phase reduction of noisy oscillators, we first introduce the second-order PRC, denoted by $\bs H(\theta)$ and defined by the Hessian of the phase map $\nabla^2 \Phi (x)$ at the point on the limit cycle associated with phase $\theta$. In Section~\ref{sec:second-order-PRC}, we will discuss the second-order PRC in details. 

\begin{proposition}[\bit{Phase reduction of noisy oscillators}]\label{prop:phase:reduction:noisy:oscillator}
For the noisy oscillator~\eqref{eq:noisy-oscillator}, the dynamics of the phase in a neighborhood of the limit cycle is 
\begin{align}\label{eq:full-SDE}
 d \theta =    \Big( \omega + \frac{\sigma^2}{2} \mathrm{tr}\big[K(\gamma(\theta))^\top  \bs H(\theta) K(\gamma(\theta))\big] \Big) dt
 + \sigma \bs Z(\theta)^\top K(\gamma(\theta)) d  W(t).
\end{align}
\end{proposition}
\begin{proof}
since we interpret  \eqref{eq:noisy-oscillator}  in It\^{o} sense, 
 we apply the It\^{o} chain rule~\cite[Theorem 4.16]{bjork2009arbitrage} to the phase map $\Phi(x(t))$ to obtain
\begin{align*}
d \Phi (x)&= \nabla \Phi(x) \cdot F(x) dt  + \sigma \nabla \Phi(x) \cdot K(x) d W(t) \nonumber \\
& \quad  + \frac{1}{2} \big(F(x)dt + \sigma K(x) d W(t) \big)^\top \nabla^2 \Phi(x)  
  \big(F(x) dt + \sigma K(x) d  W(t)  \big) \nonumber\\
 &\approx \nabla \Phi (x^\gamma)\cdot  F(x^\gamma) dt  + \sigma \nabla \Phi (x^\gamma)\cdot K(x^\gamma) dW(t)\nonumber \\
& \quad  + \frac{1}{2} \big( F( x^\gamma)dt + \sigma K( x^\gamma) d  W(t) \big)^\top \nabla^2 \Phi( x^\gamma)  
  \big( F(x^\gamma) dt + \sigma K (x^\gamma) d W(t)  \big) \nonumber\\
& = \Big( \omega + \frac{\sigma^2}{2}\mathrm{tr}\big[K( x^\gamma)^\top  \nabla^2 \Phi( x^\gamma) K( x^\gamma)\big] \Big) dt + \sigma \nabla \Phi( x^\gamma) \cdot K( x^\gamma) d W(t),
\end{align*}
which yields the desired result similar to~\cite[Theorem 4.16]{bjork2009arbitrage}. 
\end{proof}

Note that the phase equation derived in Proposition \ref{prop:phase:reduction:noisy:oscillator} is different from the existing phase equations (using PRC approach), e.g., \cite{2011_Ermentrout_Beverlin_Troyer_Netoff} and \cite{teramae2004robustness, teramae2009stochastic}. Our approach to compute $d\theta$ is to interpret the SDEs in It\^{o} and employ It\^{o}  chain rule, while in the literature, the SDEs are interpreted in Stratonovich and after computing  $d\theta$ using ordinary chain rules, they are transferred to SDEs in the It\^{o} sense. In an Appendix we provide a few examples which explain the difference between the two approaches. 

\begin{remark}
 If $K(x)$ is the identity,~\eqref{eq:full-SDE} reduces to
\begin{equation}\label{eq:SDE-oscillator}
d \theta = \Big( \omega + \frac{\sigma^2}{2} \sum_{i=1}^n \bs H_{ii}(\theta) \Big) dt + \sigma \bs Z(\theta)^\top d  W(t),
\end{equation}       
where $\bs H_{ii}$s are the diagonal entries of $\bs H$. 
\end{remark}

\begin{example}(\bit{Phase reduction of noisy Hopf bifurcation normal form}.)
We consider the normal form of {a} supercritical Hopf bifurcation with additive noise:
\begin{align}\label{eq:hopf}
\begin{split}
d x_1 &= (\mu x_1 - \omega x_2 - (x_1^2 +x_2^2)x_1) dt + \sigma dW_1(t),\\
d x_2 &= (\omega x_1 + \mu x_2 - (x_1^2 +x_2^2)x_2) dt + \sigma dW_2(t).
\end{split}
\end{align}
Recall that, without noise, the dynamics of~\eqref{eq:hopf} yields a circular limit cycle  centered at the origin with radius $\sqrt{\mu}$ and frequency $\omega$. Furthermore, given the phase $\theta$ on the limit cycle, the associated point $(x_1^\gamma, x_2^\gamma) = (\sqrt{\mu} \cos(\theta), \sqrt{\mu} \sin(\theta))$. Equivalently, $\theta = \Phi(x_1^\gamma,x_2^\gamma)=\tan^{-1}(x_2^\gamma/x_1^\gamma)$. It follows immediately that   
\begin{align}\label{eq:Z-H-Hopf}
\begin{split}
\bs Z(\theta) =  \frac{1}{\sqrt \mu}\left(\begin{array}{cc} 
-\sin(\theta) \\ \cos(\theta) 
\end{array}\right), \quad
\bs H(\theta)= \frac{1}{\mu}\left(\begin{array}{cccc}
\sin(2\theta)& -\cos(2\theta)  \\ -\cos(2\theta)& -\sin(2\theta)
\end{array}\right). 
\end{split}
\end{align}

Therefore, using \eqref{eq:SDE-oscillator}, the phase reduction for the noisy Hopf bifurcation normal form~\eqref{eq:hopf} is 
\begin{equation*}\label{eq:phase-sde-hopf}
d \theta = \omega  dt + \frac{\sigma}{\sqrt{\mu}} dW(t). 
\end{equation*} 
\end{example}

\subsection{Computation of second-order phase response curves}\label{sec:second-order-PRC}

The Hessian of the phase map, denoted by  $\bs {H}(\theta(t)) = \nabla^2\Phi(x(t))$ and called the second-order PRC,  plays an important role in phase reduction of stochastic oscillators \cite{2018_Giacomin_Poquet_Shapira, 2018_Bressloff_MacLaurin, 2019_Aminzare_Holmes_Srivastava}.  Similar to the first-order PRC, the second-order PRC can be computed using a {\em direct method} as follows. 
 \begin{equation*}\label{eq:2nd-PRC-direct}
 \bs H_{ij}(\theta) = \frac{\partial^2 \Phi}{\partial x_i \partial x_j} ( x)=
 \lim_{r\to0} \frac{1}{r} \left(\dfrac{\partial \Phi}{\partial  x_j}( x + r \hat{ i}) -\dfrac{\partial \Phi}{\partial  x_j}( x )\right).
 \end{equation*}
Note that for an identity {diffusion}  matrix $K$, one only needs to compute the diagonal 
entries of $\bs H$, as follows (without computing the first-order PRC). 
  \begin{equation*}\label{eq:2nd-PRC-diagonal-direct}
\bs H_{ii}(\theta) 
 =\lim_{r\to0} \frac{1}{r}\left(\Phi( x + r \hat{ i}) - 2 \Phi ( x) +\Phi( x - r\hat{ i})\right).
 \end{equation*}
 
 Alternatively, the second-order PRC  solves the following ODE in reverse time
 \begin{equation}\label{eq:adjoint-method-2ndPRC}
 \dot{\bs H}(\theta) =-\nabla^2{F} (\bs Z(\theta)\otimes { I})- D  F^\top \bs H(\theta) - \bs H(\theta) D  F,
 \end{equation}
where all the  arguments  $( x^\gamma(t))$ from $\nabla^2 F$ and $D F$ are dropped;  
$\nabla^2 F = [\nabla^2 F_1 \cdots \nabla^2 F_n]$ is an $m\times m^2$ matrix and represents the Hessian matrix of the vector field $ F$, $\otimes$ is the Kronecker product and $ I$ is the $m\times m$ identity matrix. 
The  initial condition is determined by the following constraint 
\begin{equation}\label{eq:adjoint-method-constraint-2ndPRC}
 F^\top( x^\gamma(0))  \bs H(\theta)   F( x^\gamma(0)) = - \left( F^\top D  F^\top \right)( x^\gamma(0))\bs Z(\theta).  
 \end{equation}
 
\begin{proposition}[\bit{Computing the second-order PRC}]\label{prop:PRC}
Consider system \eqref{eq:oscillator-model} with an asymptotically stable limit cycle $ x^\gamma(t)$ and its corresponding phase map $\Phi$. Let $\nabla^2 \Phi$ be the Hessian matrix of the phase map $\Phi$.  Then $\nabla^2 \Phi$ solves \eqref{eq:adjoint-method-2ndPRC} with constraint \eqref{eq:adjoint-method-constraint-2ndPRC}. \oprocend
\end{proposition}

A proof of Proposition~\ref{prop:PRC} can be found in~\cite[Section 2]{wilson2018greater}. 

\begin{remark}\label{initial-conditions-PRC-2ndPRC}
	{The following choices of initial conditions for 
	$\nabla\Phi$ and $\nabla^2\Phi$ guarantee the constraints given in \eqref{eq:adjoint-method-constraint} and \eqref{eq:adjoint-method-constraint-2ndPRC}, respectively:
		\begin{align*}
		\nabla\Phi( x^\gamma(0)) &= \dfrac{\omega}{( F^\top F)( x^\gamma(0))} F( x^\gamma(0)),\\
		\nabla^2\Phi ( x^\gamma(0)) 
		&= - \dfrac{\omega}{2( F^\top F)( x^\gamma(0))}(D F+D F^\top)( x^\gamma(0)).
		\end{align*}
	}
\end{remark} 

\begin{remark} 
	In what follows, we vectorize~\cite{macedo2013typing} equation~\eqref{eq:adjoint-method-2ndPRC} and combine the corresponding equations of the first- and second-order PRC. 
	Let $\bs H_v= \text{vec}(\bs H)$ be the vectorization of $\bs H$, i.e., the vector of the columns of $\bs H$. Then,  
\begin{align}\label{eq:vector-PRC-2ndPRC}
\begin{split}
\dot{\bs Z}&= - D F^\top \bs Z,\\
\dot{\bs{H}}_v&= - ( I \otimes D F^\top+D F^\top\otimes I ) \bs{H}_v- ( I \otimes \nabla^2 F)\text{vec}({\bs Z\otimes I}),\\
&= - ( I \otimes D F^\top+D F^\top\otimes I ) \bs{H}_v - ({\bs Z^\top\otimes I\otimes  I})\text{vec}(\nabla^2 F), 
\end{split}
\end{align}
with constraints
	\begin{align*}
      F^\top ( x^\gamma(0))\bs Z(0) &= \omega,\\
     ( F^\top\otimes   F^\top) ( x^\gamma(0)) \bs{H}_v (0)&=\left( F^\top D  F^\top \bs Z \right)( x^\gamma(0)), 
	\end{align*}
	where  the following vectorization  equalities are used for arbitrary matrices $A$, $B$, and $C$:
	\begin{align*}
     	\text{vec}(AB)&=( I \otimes A)\text{vec}(B) =(B^\top\otimes  I )\text{vec}(A),\\
		\text{vec}(ABC)&=(C^\top \otimes A)\text{vec}(B). 
	\end{align*}
	Here $ I $ is an identity matrix of the appropriate size. 
\end{remark}

\begin{remark} 
		Due to the negative sign in the right hand side of~\eqref{eq:adjoint-method-2ndPRC}, or equivalently \eqref{eq:vector-PRC-2ndPRC},  its stability is the opposite of the stability of the limit cycle. Hence, the  equation needs to be solved in reverse time. 
\end{remark}

\begin{example}(\bit{Hopf bifurcation normal form.})
For the Hopf bifurcation dynamics~\eqref{eq:hopf} with $\sigma=0$, 
it can be verified that the matrix $\bs H$ derived in \eqref{eq:Z-H-Hopf}
satisfies \eqref{eq:adjoint-method-2ndPRC}. 
\end{example}

 \begin{example}(\bit{Van der Pol oscillator.})\label{example:vanderpol1}
We now consider the van der Pol oscillator with additive white noise
    \begin{subequations}\label{eq:vanderpol}
    \begin{align}
            d x_1 & = \left(x_1 - \frac{1}{3} x_1^3 -x_2\right) dt + \sigma dW_1(t) \\
            d x_2 &=x_1 dt + \sigma dW_2(t) .
    \end{align}
    \end{subequations}
   Figure~\ref{fig:vanderpol_PRC} shows the first-order PRC and the second-order PRC for dynamics~\eqref{eq:vanderpol} with $\sigma=0$. These 2-component PRCs are computed by numerically solving~\eqref{eq:vector-PRC-2ndPRC} with initial conditions discussed in Remark~\ref{initial-conditions-PRC-2ndPRC}.
   
    \begin{figure}[ht!]
        \centering
        \includegraphics[width=0.3\linewidth]{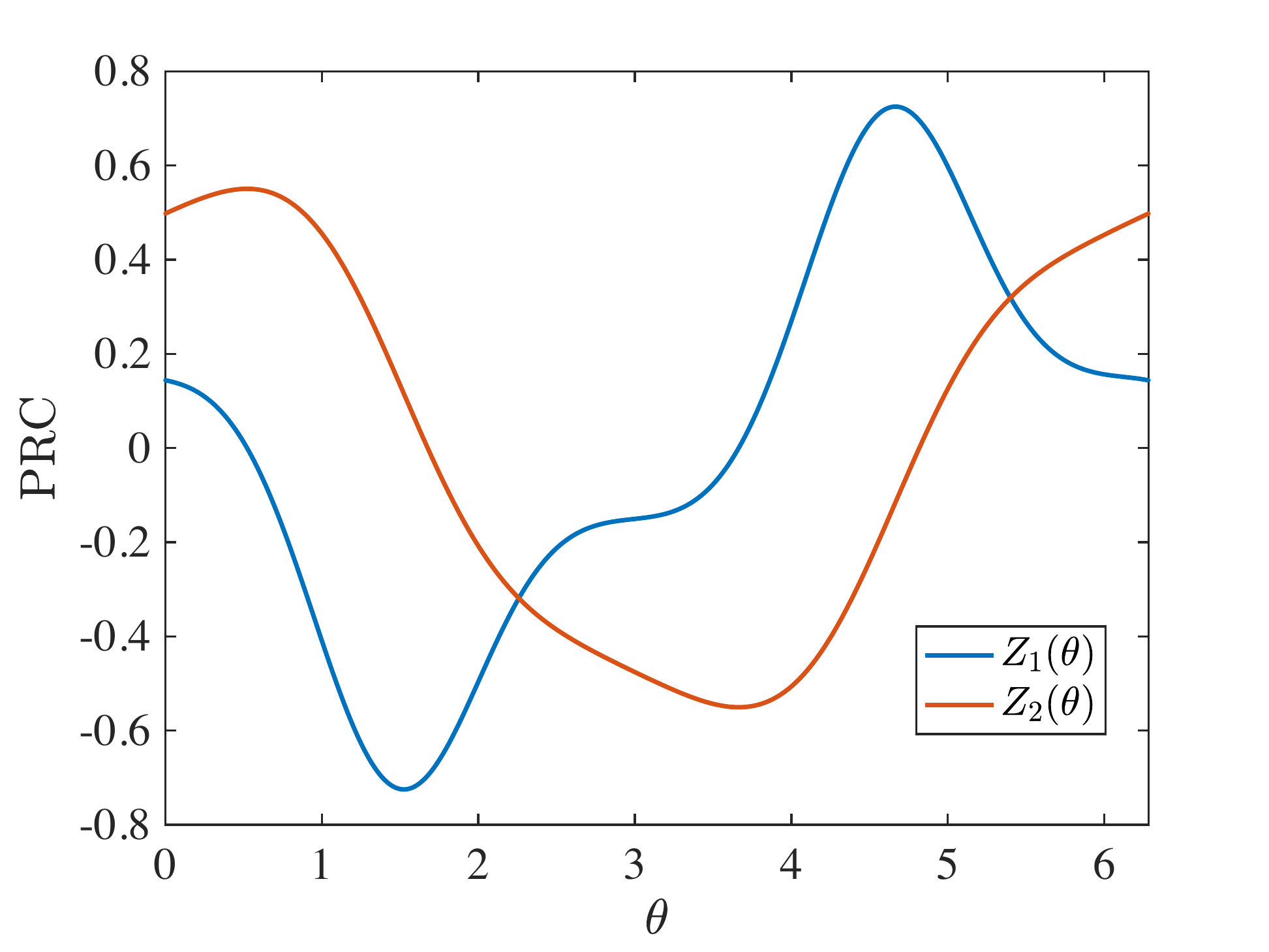}
        \includegraphics[width=0.3\linewidth]{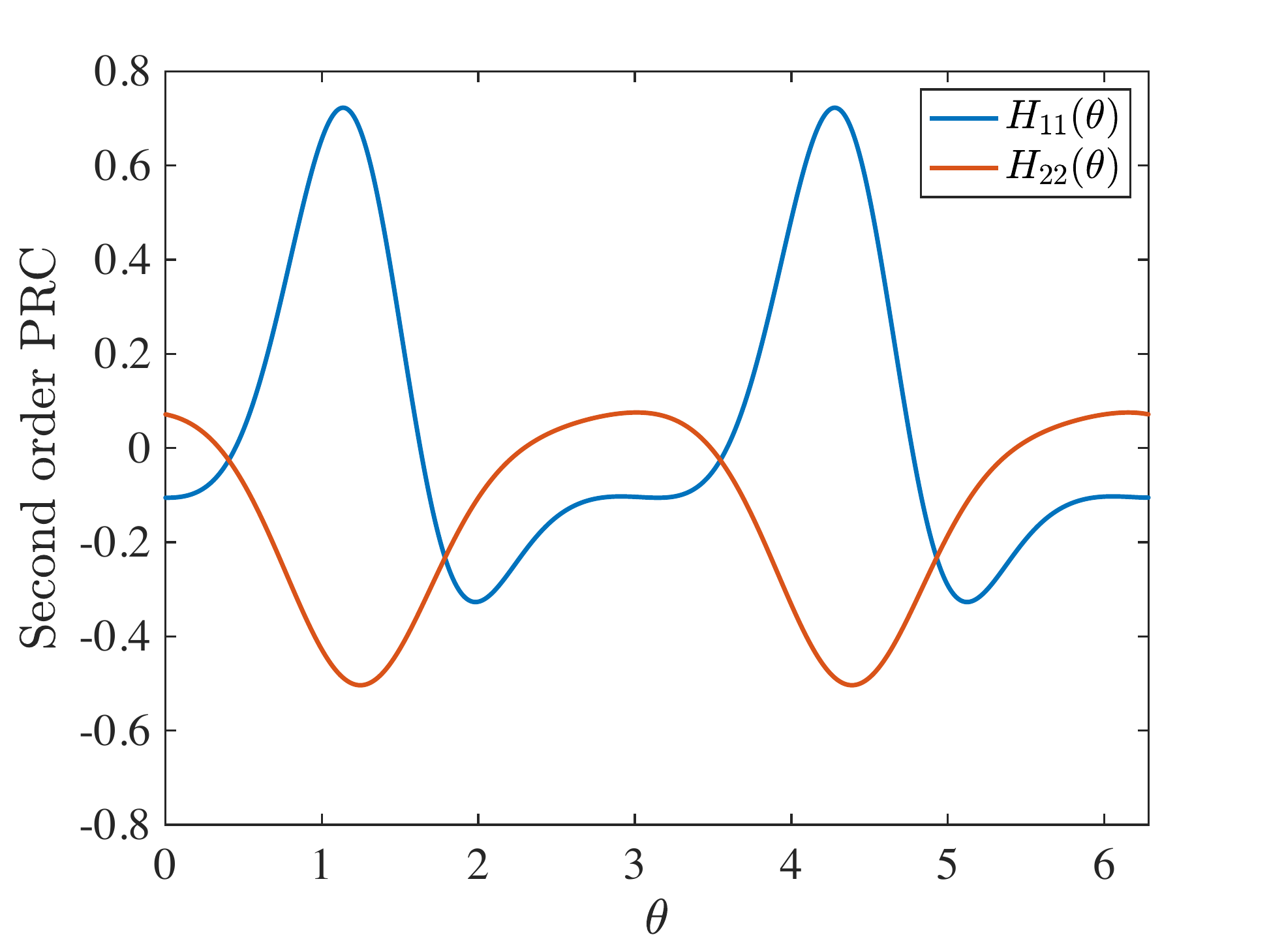}\\
        \caption{{The components of the first-order PRC (left) and the second-order PRC (right) of the} Van der Pol oscillator.}
        \label{fig:vanderpol_PRC}
    \end{figure}
    
\end{example}

\section{Phase reduction of weakly coupled noisy oscillators}\label{sec:coupled-phase-reduction}

In this section, we derive coupled phase equations of a network of noisy oscillators which are weakly connected through noisy interactions. 
For $i=1,\ldots, N$, let
\begin{equation}\label{eq:oscillator_coupled_noisy}
d{ x_i} =  \underbrace{ F ( x_i) dt+ \sigma \B( x_i) d  W_i(t)}_{\text{intrinsic dynamics}} \;+\; \underbrace{\sum_{j=1}^N c_{ij}\big(\epsilon \H( x_j,  x_i) dt +
\delta \C( x_j,  x_i) d  W_{ij}(t)\big)}_{\text{coupling dynamics}}
\end{equation}
describe the dynamics of each oscillator $i$ and its interaction with its adjacent oscillators, the set of oscillators which are connected to $i$. 
The set up of \eqref{eq:oscillator_coupled_noisy} is very similar to \eqref{eq:general-form2}:
There is an underlying weighted graph with weight $c_{ij}$ which does \textit{not} need to be undirected, i.e., $c_{ij} \neq c_{ji}$.   

 The first two terms describe the dynamics of each isolated oscillator as discussed in \eqref{eq:noisy-oscillator}. We assume that the oscillators have identical dynamics and allow small common noise on their dynamics. The noise can arise from a noisy environment, which affects the dynamics of each individual, or the intrinsic property of the agents. In terms of modeling, one may assume that the model parameters are buried in noise. 
The last two terms of \eqref{eq:oscillator_coupled_noisy} describe the noisy interactions between oscillator $i$ and its adjacent oscillators. This noise arises from the edges of the graph. 

In \eqref{eq:oscillator_coupled_noisy}, the state variable $ x_i$, the internal dynamics $ F ( x_i)$, and the deterministic interaction function  $ \H( x_j, x_i)$ are $m-$dimensional vectors. 
The $m\times m$ diffusion matrices $\sigma \B$ and $\delta \C$ 
 describe the correlation of {\it common} noise and  {\it interaction} noise, respectively.  The vectors 
  $d W_i$ and   $d W_{ij}$ are $m-$dimensional standard Wiener process increments, 
   i.e., $\langle d W_i^k(t), d W_i^l(s)\rangle = \delta_{kl}(t-s)$, where $d W_i^k$ is the $k-$th argument of $d W_{i}$. Similarly,  $\langle d W_{ij}^k(t), d W_{ij}^l(s)\rangle = \delta_{kl}(t-s)$, where $d W_{ij}^k$ is the $k-$th argument of $d W_{ij}$.

The constant parameters $\epsilon$, $\sigma$, and $\delta$, which  respectively represent the deterministic coupling strength, the  common and interaction noise intensities, are assumed to be sufficiently small and satisfy 
\begin{equation}\label{order-coupling-noise}
 \mathcal{O}(\sigma) = \mathcal{O}(\delta)= \mathcal{O}(\sqrt\epsilon), \qquad 0<\epsilon\ll1. 
 \end{equation}

For $\epsilon = \sigma = \delta =0$, \eqref{eq:oscillator_coupled_noisy} reduces to $\dot{ x}_i=  F( x_i)$, that we assume admits an asymptotically stable limit cycle $ x_i^\gamma$ with frequency $\omega$. 
 We denote $ x_i^\gamma$ by $\gamma(\theta_i(t))  = \gamma(\omega t+\psi_i(t))$, where $\theta_i$ and $\psi_i$ are respectively the asymptotic phase and relative phase of oscillator $i$. 

A solution of~\eqref{eq:oscillator_coupled_noisy}, denoted by 
$ x_i(t)$, is called a {\it noisy oscillator}.   
Note that we assume that  $\epsilon$,  $\sigma$, and $\delta$ are  small enough so that the trajectories stay within the basin of attraction of the limit cycles $ x_i^\gamma$ after receiving the deterministic and stochastic perturbations.
We  also assume that $ x_i$ can be approximated by $ x_i^\gamma$.  

\begin{proposition}[\bit{Phase reduction of weakly coupled noisy oscillators}]
For the  noisy coupled  oscillators~\eqref{eq:oscillator_coupled_noisy}, the dynamics of the coupled phases $\theta_i$ in  neighborhoods of the limit cycles are
\begin{align}\label{eq:full-couplde-SDE}
 d \theta_i & =  \left(\omega+\epsilon\sum_j c_{ij} \bs Z^\top \H +
  \frac{\sigma^2 }{2}  \mathrm{tr}[\B^\top \bs H \B]+  \frac{\delta^2 }{2} \mathrm{tr}\Big[\sum_j c_{ij}^2 \C^\top \bs H \C\Big] \right)\;dt \nonumber
\\
 &\quad+  \sigma \bs Z^\top \B \;d  W_i(t) + \delta\sum_j c_{ij} \bs Z^\top\C \;d  W_{ij}(t), 
\end{align}
where $\bs Z=\bs Z(\theta_i)$,  $\bs H= \bs H(\theta_i)$, $\H= \H(\gamma(\theta_j), \gamma(\theta_i))$, 
$\B=\B(\gamma(\theta_i))$,  $\C=\C(\gamma(\theta_j), \gamma(\theta_i))$, and $\sum_j=\sum_{j=1}^N$. 
\end{proposition}

Note that we use $\H$ for the coupling function and $\bs H$ for the second-order PRC. 
\begin{proof}
We apply the It\^{o} formula~\cite[Theorem 4.16]{bjork2009arbitrage} to the phase map $\Phi( x_i(t))$. Then
\begin{align*}
d \Phi ( x_i)&= \nabla \Phi  ( x_i)\cdot  \Big[ F( x_i)+\epsilon \sum_j c_{ij}\H( x_j,  x_i)\Big] dt
\\
 &\quad+ \nabla \Phi( x_i) \cdot \Big(\sigma\B( x_i) d  W_i(t) + \delta\sum_j c_{ij}\C( x_j,  x_i) d  W_{ij}(t)\Big)  
 \\
& \quad + \frac{1}{2} \Big[  F( x_i) dt+\epsilon \sum_j c_{ij}\H( x_j,  x_i)dt + \sigma\B( x_i) d  W_i(t) + \delta\sum_j c_{ij} \C( x_j,  x_i) d  W_{ij}(t) \Big]^\top\\
 &\qquad \nabla^2 \Phi( x_i) \Big[  F( x_i) dt+\epsilon \sum_j c_{ij}\H( x_j,  x_i)dt + \sigma\B( x_i) d  W_i(t) + \delta\sum_j c_{ij} \C( x_j,  x_i) d  W_{ij}(t) \Big]
 \\
 &= \nabla \Phi ( x_i)\cdot  \Big[ F( x_i)+\epsilon\sum_j c_{ij}\H( x_j,  x_i)\Big] dt 
 \\
 &\quad+ \nabla \Phi( x_i) \cdot \Big(\sigma\B( x_i) d  W_i(t) + \delta\sum_j c_{ij}\C( x_j,  x_i) d  W_{ij}(t)\Big)  
 \\
& \quad + \frac{1 }{2}  \mathrm{tr}\Big[\sigma^2 \B( x_i)^\top\nabla^2 \Phi( x_i)\B( x_i)+ \delta^2 \sum_j c_{ij}^2\C( x_j,  x_i)^\top 
\nabla^2 \Phi(\bs x_i) \C( x_j,  x_i)\Big]dt 
\\
&\approx \nabla \Phi ( x_i^\gamma)\cdot  \Big[ F( x_i^\gamma)+\epsilon\sum_j c_{ij}  \H( x_j^\gamma,  x_i^\gamma)\Big] dt 
 \\
 &\quad+ \nabla \Phi( x_i^\gamma) \cdot \Big(\sigma\B( x_i^\gamma) d  W_i(t) + \delta\sum_j c_{ij} \C( x_j^\gamma,  x_i^\gamma) d  W_{ij}(t)\Big)  
 \\
& \quad + \frac{1 }{2}  \mathrm{tr}\Big[\sigma^2 \B( x_i^\gamma)^\top\nabla^2 \Phi( x_i^\gamma)\B( x_i^\gamma)+ \delta^2 \sum_j c_{ij}^2\C( x_j^\gamma,  x_i^\gamma)^\top 
\nabla^2 \Phi( x_i^\gamma) \C( x_j^\gamma,  x_i^\gamma)\Big]dt 
\\
&=  \Big[\omega+\epsilon\sum_j c_{ij}\nabla\Phi( x_i^\gamma)\cdot \H( x_j^\gamma,  x_i^\gamma)\Big] dt 
 \\
 & \quad + \frac{1 }{2}  \mathrm{tr}\Big[\sigma^2 \B( x_i^\gamma)^\top\nabla^2 \Phi( x_i^\gamma)\B( x_i^\gamma)+ \delta^2 \sum_j c_{ij}^2\C( x_j^\gamma,  x_i^\gamma)^\top 
\nabla^2 \Phi( x_i^\gamma)  \C( x_j^\gamma,  x_i^\gamma)\Big]dt 
\\
 &\quad+ \nabla \Phi( x_i^\gamma) \cdot \Big(\sigma\B( x_i^\gamma) d  W_i(t) + \delta\sum_j c_{ij} \C( x_j^\gamma,  x_i^\gamma) d  W_{ij}(t)\Big),  
 \end{align*}
which yields the desired result. The first equality is obtained by the It\^{o} formula, the second equality obtained from the equalities 
$dt\cdot dt = dt \cdot d  W_i =dt \cdot d  W_{ij} =d W_{ij} \cdot d  W_i =0$ and 
$d W_i \cdot d W_i=d W_{ij}\cdot d W_{ik}= \delta_{jk}dt$. The approximation is obtained by the assumption that each noisy oscillator $ x_i$ can be  approximated by $ x_i^\gamma$, and the last equality obtained from $\nabla \Phi( x_i^\gamma)\cdot  F( x_i^\gamma) =\omega$.
\end{proof}

\section{Stochastic synchronization of coupled phase reduced equations}\label{sec:synchronization-coupled-phase}

In this section, we apply Theorems  \ref{thm:cond:synch2} and  \ref{thm:cond:synch3} (and similarly  Theorem \ref{thm:cond:synch4}) to the coupled phase equations \eqref{eq:full-couplde-SDE} and, based on the given PRCs, $\bs Z$ and $\bs H$, we will design appropriate $\B, H,$ and $\C$  that guarantee the oscillators'  synchronization. 

In what follows, we consider three separate cases: (1) no edge coupling, (2) deterministic edge coupling, and (3) stochastic edge coupling. 

\textbf{Case 1. No edge coupling.} First, we consider $N$ oscillators which are connected only through a common noise, i.e., consider \eqref{eq:full-couplde-SDE} with no  coupling ($c_{ij} = c_{ji}=0$) and  $dW_i=d\hat W$, where $d\hat W$ is an $m-$dimensional Wiener increment. 
\begin{align}\label{eq:full-couplde-SDE-m-dim}
 d \theta_i & =  
 \left(\omega +
  \frac{\sigma^2 }{2}  \mathrm{tr}[\B^\top \bs H \B]\right)\;dt +\sigma \bs Z^\top \B \;d  \hat W(t),
\end{align}
and check  conditions (i) and (iv) of Theorem \ref{thm:cond:synch2} and condition (ii) of Theorem \ref{thm:cond:synch3}. To apply these theorem,  \eqref{eq:full-couplde-SDE-m-dim} must be in the format of \eqref{eq:general-form2}, with a $1-$dimensional Wiener increment. To this end, we write $\sigma \bs Z^\top \B \;d  \hat W(t)$ as $\sigma \ino \;d W(t)$ where 
$\ino = \sqrt{\bs Z^\top \B\B^\top\bs Z}$ is a scalar and $dW$ is
 a $1-$dimensional Wiener increment. 
\begin{align}\label{eq:full-couplde-SDE-2}
 d \theta_i & =  
 \left(\omega +
  \frac{\sigma^2 }{2}  \mathrm{tr}[\B^\top \bs H \B]\right)\;dt +\sigma \ino \;d W(t). 
\end{align}

\begin{description}[leftmargin=*]
\item[Condition i of Theorem \ref{thm:cond:synch2}.] 
Assume that $K$ is differentiable and let 
\begin{equation}\label{barcF-no-coupling}
\bar c_{\mathcal{F}}  := \dfrac{\sigma^2}{2} {\sup_{(\phi,t) \in (-\frac{\pi}{2}, \frac{\pi}{2})\times[0,T)}} \lambda_{\max} \left[\frac{\partial }{\partial \phi}  \Big(\B(\gamma(\omega t+\phi))^\top  \bs H(\omega t+\phi) \B(\gamma(\omega t+\phi))\Big)\right], 
\end{equation}
where $\lambda_{\max}[A]$ denotes the maximum eigenvalue of $A$.  Then, for $\phi, \eta\in {(-\frac{\pi}{2}, \frac{\pi}{2})}$ and $t\in[0,T)$:
\begin{equation*}
(\phi-\eta) (\mathcal{F}(\phi,t)-  \mathcal{F}(\eta,t))
\leq  \bar c_{\mathcal{F}}(\phi-\eta)^2, 
\end{equation*}
where $\mathcal{F}(\phi,t)= \omega +
  \frac{\sigma^2 }{2}  \mathrm{tr}[\B(\gamma(\omega t+\phi))^\top \bs H(\omega t+\phi) \B(\gamma(\omega t+\phi))].$ 
  We used $\mathrm{tr}[A] = \sum_{i=1}^N e_i^\top A e_i$,  where $e_i$s are the standard basis of $\mathbb R^N$, and 
$e_i^\top A e_i \leq \lambda_{\max}[A] e_i^\top e_i.$

\item[Condition iv of Theorem \ref{thm:cond:synch2} and condition ii of Theorem \ref{thm:cond:synch3}.] 
Assume that $\B$ is differentiable. Then 
$\ino(\phi,t) = \sqrt{\bs Z(\omega t+\phi)^\top (\B\B^\top)(\gamma(\omega t+\phi))\bs Z(\omega t+\phi)}$
becomes differentiable and  for any $\phi,\eta \in (-\frac{\pi}{2}, \frac{\pi}{2})$ and $t\in[0,T)$, 
\begin{equation}\label{barcK-no-coupling}
|\ino(\phi,t)-\ino(\eta,t)| \leq   \bar c_{\ino}  |\phi-\eta|,
 \quad \text{where}\quad
\bar c_{\ino} :=  
{\sup_{(\phi,t) \in (-\frac{\pi}{2}, \frac{\pi}{2})\times[0,T)}}
\frac{\partial }{\partial \phi} \ino(\phi,t),
\end{equation}
and 
\begin{equation}\label{underbarcK-no-coupling}
 \underbar c_{\ino}  (\phi-\eta)^2 \leq (\phi-\eta) (\ino(\phi,t)-\ino(\eta,t)),
 \quad \text{where}\quad
 \underbar c_{\ino} :=  {\inf_{(\phi,t) \in (-\frac{\pi}{2}, \frac{\pi}{2})\times[0,T)}}
\frac{\partial }{\partial \phi}  \ino(\phi,t). 
\end{equation}

\end{description}

In summary, we proved the following proposition. 

\begin{proposition}\label{prop:noise-induced-sync}
Consider  \eqref{eq:full-couplde-SDE-2} and assume $\B$ is differentiable. Then \eqref{eq:full-couplde-SDE-2} stochastically synchronizes if 
\[-2\bar c_{\mathcal{F}} + \sigma^2 (2  \underbar c_{\ino}^2 - \bar c_{\ino}^2) 
= \sigma^2(2  \underbar c_{\ino}^2 - \bar c_{\ino}^2-\bar\lambda_{\mathcal F})>0,\]
where $\bar\lambda_{\mathcal F} = \frac{2}{\sigma^2} \bar c_{\mathcal{F}},$ and the constant bounds are defined in \eqref{barcF-no-coupling}-\eqref{underbarcK-no-coupling}.

\end{proposition}

Note that for an appropriate choice of  $\alpha$, the corresponding level set of 
 the Lyapunov function  $\frac{1}{2}e^\top e$, denoted by $\mathscr{L}_{\alpha} := \{\phi_i\;|\: \frac{1}{2}e^\top e = \alpha\}$ becomes a subset of  $(-\frac{\pi}{2}, \frac{\pi}{2})^N$. Therefore, since the Lyapunov function is decreasing,  $\mathscr{L}_{\alpha}$ becomes an invariant set and the choice of the Lyapunov function remains valid. 

\begin{example}
	We consider three Van der Pol oscillators~\eqref{eq:vanderpol} subject to the common noise $W_1$ and $W_2$ with intensity $\sigma =0.5$. The simulation results are shown in Figure~\ref{fig:vanderpol-noise}. The state of oscillator $i$ is denoted by $(x_1^i, x_2^i)$.
	The applied noise is removed once synchronization is achieved. In this example, $K$ is chosen as an identity matrix and thus $\mc F  = \omega + \frac{\sigma^2}{2} \mathrm{tr}(H)$ and $\mc K = \sqrt{\bs{Z}^\top \bs{Z}}$. It can be verified that condition of Proposition~\ref{prop:noise-induced-sync} is satisfied in a neighborhood of $\phi=0$ but it is not satisfied for every $\phi \in (-\frac{\pi}{2}, \frac{\pi}{2})$. This illustrates that synchronization may be achieved beyond regimes obtained by the sufficient conditions in Proposition~\ref{prop:noise-induced-sync}.

	\begin{figure}[ht!]
		\subfigure[Noise-induced synchronization]{
			\includegraphics[width=0.45\textwidth]{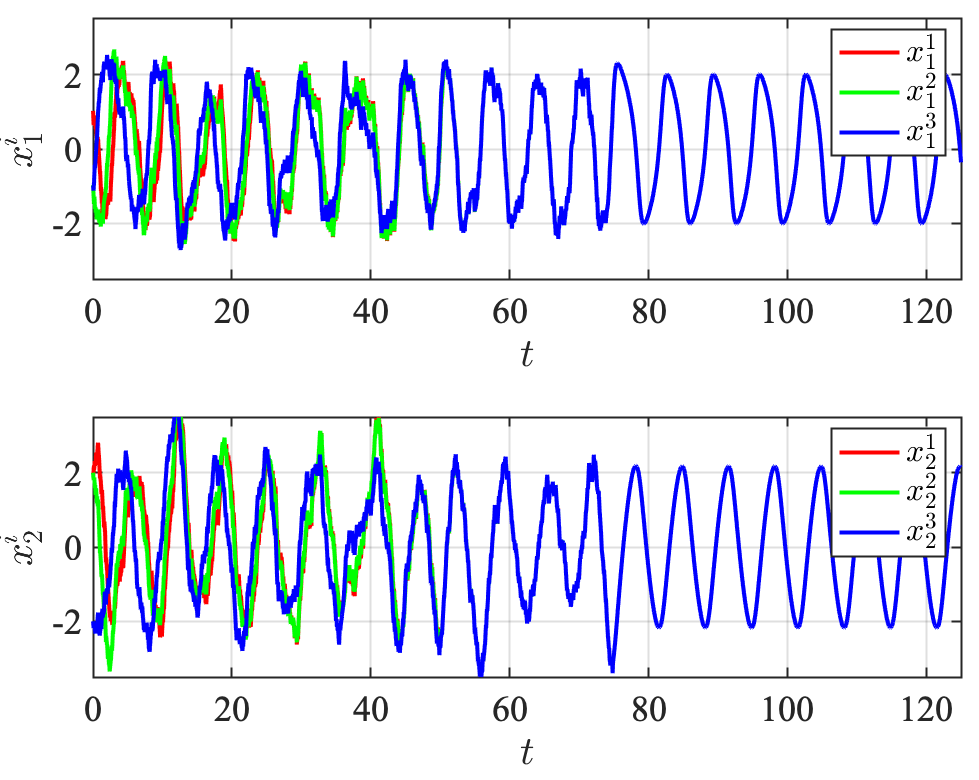}	
			\label{fig:vanderpol-noise}
		}
			\subfigure[Coupling-induced synchronization]{
		\includegraphics[width=0.45\textwidth]{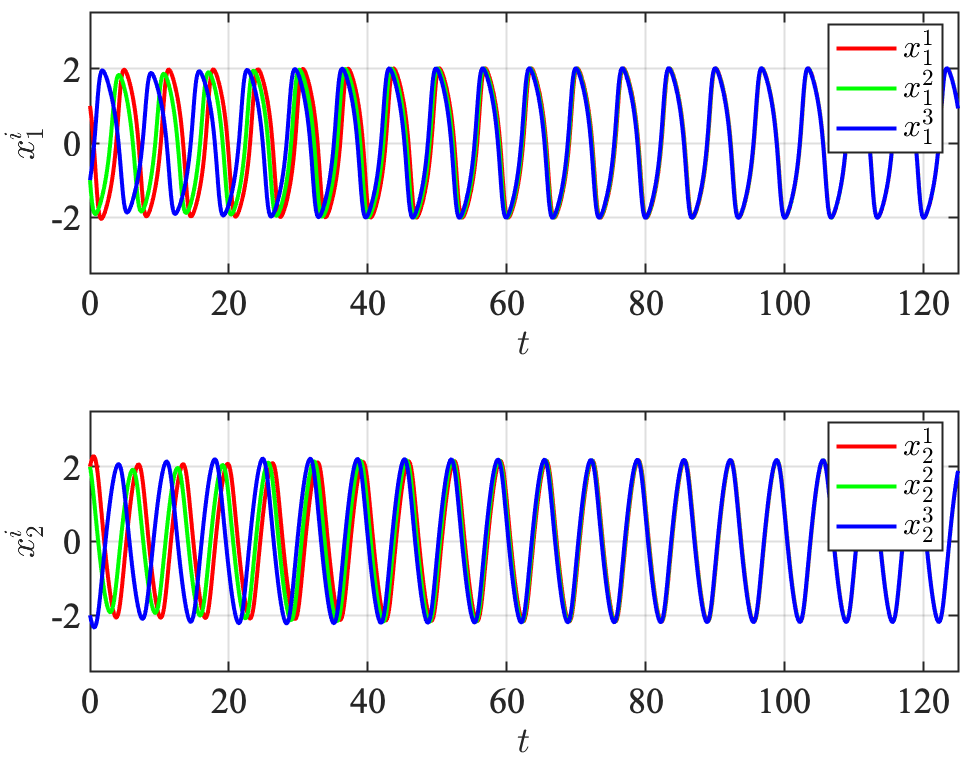}	\label{fig:vanderpol-coupling}
	}
		\caption{Synchronization of three Van der Pol Oscillators. (a) Three oscillators subject to the same common noise. Noise is removed once the synchronization is achieved. (b) Three coupled Van der Pol oscillators interacting on a line graph with diffusive coupling.    \label{fig:example-sync-vanderpol}}
	\end{figure}
\end{example}

\vskip .2in 

\textbf{Case 2. Deterministic coupling and deterministic averaging theory.} Next, we consider $N$ oscillators which are connected through a common noise and a deterministic coupling among themselves, i.e., consider \eqref{eq:full-couplde-SDE} with nonzero coupling $\H$, zero stochastic coupling $\C$, and as discussed in Case 1, we let $\sigma\ino \;dW$ describes the common noise:
\begin{align}\label{eq:couplde-SDE-no-delta}
 d \theta_i & =  
 \left(\omega +
  \frac{\sigma^2 }{2}  \mathrm{tr}[\B^\top \bs H \B] +\epsilon\sum_j c_{ij} \bs Z^\top \H\right)\;dt +\sigma\ino \;dW(t). 
\end{align}
Note that in order to use the results of Theorems \ref{thm:cond:synch2}-\ref{thm:cond:synch4}, the coupling function $\bs Z^\top \H$ must satisfy $\bs Z(x)^\top \H(y,x)= - \bs Z(y)^\top \H(x,y)$. In what follows, we show that if we use the Averaging Theory (\cite[Theorem 4.1.1]{guckenheimer2013nonlinear}), we can approximate $\bs Z(x)^\top \H(y,x)$ by $\mathcal{H}(y-x)$ ($\mathcal H$ will be defined in \eqref{barH} below). Then, by  assuming  
$\H(y,x)= - \H(x,y)$, we can conclude that 
$\mathcal{H}(y-x) = - \mathcal{H}(x-y)$, as desired. 

For any $i=1,\ldots, N$, let 
\begin{align}\label{eq:phi}
\phi_i (t)=
 \theta_i(t) -
  \omega t - 
  \frac{\sigma^2 }{2} \int_0^t \mathrm{tr}[\B^\top \bs H \B](\theta_i(\tau)) \;d\tau -
  \sigma \int_0^t \ino(\theta_i(\tau)) \;d  W(\tau). 
\end{align}
Then, for any $i$,
\begin{align}\label{eq:d-phi}
 d \phi_i(t) & =  \epsilon\sum_j c_{ij}\bs Z(\theta_i(t))^\top \H(\gamma(\theta_j(t)), \gamma(\theta_i(t))). 
 \end{align}
Note that  $d \phi_i(t)$ is of order $\epsilon$, and in \eqref{order-coupling-noise}, we assumed that $\mathcal{O}(\epsilon) = \mathcal{O}(\sigma^2)$. Therefore, to keep   $d \phi_i(t)$ of order $\epsilon$, we approximate the arguments $\theta_i$ and $\theta_j$ in the right hand side of $d \phi_i(t)$ by $\omega t + \phi_i$ and $\omega t + \phi_j$, respectively, and ignore the terms of order $\sigma$ and $\sigma^2$:
\begin{align}\label{eq:phi:before:averaging}
 d \phi_i(t) & =  \epsilon\sum_j c_{ij}\bs Z(\omega t+\phi_i(t))^\top \H(\gamma(\omega t+\phi_j(t)), \gamma(\omega t+\phi_i(t))). 
 \end{align}
Applying Averaging Theory to \eqref{eq:phi:before:averaging}, we get the following approximation of order $\epsilon$ for $ d \phi_i(t)$:
\begin{align}\label{eq:d-phi:after:averaging}
 d \phi_i(t) & =  \epsilon\sum_j c_{ij} \mathcal{H}(\phi_j-\phi_i), 
 \end{align}
 where 
\begin{align}\label{barH}
\mathcal{H}(\phi) = \dfrac{1}{2\pi}\int_0^{2\pi}  \bs Z(\xi)^\top \H(\gamma(\xi+ \phi), \gamma(\xi)) \;d\xi. 
 \end{align}
Combining \eqref{eq:couplde-SDE-no-delta}, \eqref{eq:phi}, and \eqref{eq:d-phi:after:averaging}, we get 
\begin{align*}\label{eq:couplde-SDE-no-delta-averaging}
 d \theta_i & =  
 \left(\omega +
  \frac{\sigma^2 }{2}  \mathrm{tr}[\B^\top \bs H \B] (\theta_i) +\epsilon\sum_j c_{ij} \mathcal H(\theta_j-\theta_i) \right)\;dt +\sigma \ino(\theta_i) \;d  W(t). 
\end{align*}
Note that we used 
$\epsilon \mathcal H(\psi_j-\psi_i) \approx \epsilon \mathcal H(\theta_j-\theta_i)$, since we are only interested in terms of order $\epsilon$. 

Now we check the conditions of Theorems \ref{thm:cond:synch2}-\ref{thm:cond:synch3}, where 
$\mathcal F = \omega +
  \frac{\sigma^2 }{2}  \mathrm{tr}[\B^\top \bs H \B]$ 
and $\mathcal{H}$ is as defined in \eqref{barH}. 

\begin{description}[leftmargin=*]
\item[Condition i of Theorem \ref{thm:cond:synch2}.] 
Assume that $K$ is differentiable and let 
\begin{equation}\label{barcF-deterministic-coupling}
\hat c_{\mathcal{F}} = {\sup_{\phi \in (-\frac{\pi}{2}, \frac{\pi}{2})}} \frac{N}{2}   \lambda_{\max} \left[\frac{\partial }{\partial \phi}  \left(\B(\gamma(\phi))^\top  \bs H(\phi) \B(\gamma(\phi))\right)\right], 
\end{equation}
where $\lambda_{\max}[A]$ denotes the maximum eigenvalue of $A$.  Then, for $\phi, \eta\in {(-\frac{\pi}{2}, \frac{\pi}{2})}$
\begin{equation*}
(\phi-\eta) (\mathcal{F}(\phi)-  \mathcal{F}(\eta))
\leq  \bar c_{\mathcal{F}}(\phi-\eta)^2, \quad  \bar c_{\mathcal{F}} = \sigma^2 \hat c_{\mathcal{F}}. 
\end{equation*}
We used $\mathrm{tr}[A] = \sum_{i=1}^N e_i^\top A e_i$,  where $e_i$s are the standard basis of $\mathbb R^N$, and 
$e_i^\top A e_i \leq \lambda_{\max}[A] e_i^\top e_i.$

\item[Condition ii of Theorem \ref{thm:cond:synch2}.] 
Assume that $\H( x, y) = -\H( y, x)$ and let 
\begin{equation}\label{underbarcH-deterministic-coupling}
 \ \underbar c_{\mathcal{H}}
 = \inf_{\phi{\in(-\frac{\pi}{2}, \frac{\pi}{2})}} \frac{1}{2\pi\phi^2} \int_0^{2\pi} \phi \;\bs Z(\xi)^\top \H(\gamma(\xi+\phi), \gamma(\xi))\;d\xi. 
\end{equation}
Then $\mathcal{H}(\phi)$ becomes an odd function and  $\phi\mathcal{H} (\phi) \geq  \underbar c_{\mathcal H}\phi^2.  $

\item[Condition iv of Theorem \ref{thm:cond:synch2} and condition ii of Theorem \ref{thm:cond:synch3}.] 
Assume that $\B$ is differentiable. Then 
$\ino(\phi,t) = \sqrt{\bs Z(\omega t+\phi)^\top (\B\B^\top)(\gamma(\omega t+\phi))\bs Z(\omega t+\phi)}$
becomes differentiable and  for any $\phi,\eta \in (-\frac{\pi}{2}, \frac{\pi}{2})$ and $t\in[0,T)$, 
\begin{equation}\label{barcK-deterministic-coupling}
|\ino(\phi,t)-\ino(\eta,t)| \leq   \bar c_{\ino}  |\phi-\eta|,
 \quad \text{where}\quad
\bar c_{\ino} :=  
{\sup_{(\phi,t) \in (-\frac{\pi}{2}, \frac{\pi}{2})\times[0,T)}}
\frac{\partial }{\partial \phi} \ino(\phi,t),
\end{equation}
and 
\begin{equation}\label{underbarcK-deterministic-coupling}
 \underbar c_{\ino}  (\phi-\eta)^2 \leq (\phi-\eta) (\ino(\phi,t)-\ino(\eta,t)),
 \quad \text{where}\quad
 \underbar c_{\ino} :=  {\inf_{(\phi,t) \in (-\frac{\pi}{2}, \frac{\pi}{2})\times[0,T)}}
\frac{\partial }{\partial \phi}  \ino(\phi,t). 
\end{equation}

\end{description}

In summary, we proved the following proposition. 

\begin{proposition}\label{prop:synch-deterministic-coupling}
Consider  \eqref{eq:couplde-SDE-no-delta}. Assume $\B$ is differentiable and $\H( x, y) = -\H( y, x)$. Then \eqref{eq:couplde-SDE-no-delta} stochastically synchronizes if  one of the following conditions hold:
\begin{enumerate}
\item $c: =-2\sigma^2 \hat  c_{\mathcal{F}} + {2}\epsilon \underbar c_{\mathcal{H}}  \lambda- \sigma^2 \bar c_{\ino}^2>0,$
\item $c<0$  but $c+ 2 \sigma^2 \underbar c_{\ino}^2>0,$
\end{enumerate}
where the constant bounds are defined in \eqref{barcF-deterministic-coupling}-\eqref{underbarcK-deterministic-coupling}. 
\end{proposition}

\begin{example}
		We now consider three diffusively coupled Van der Pol oscillators~\eqref{eq:vanderpol} interacting on a line graph. Let the state of oscillator $i$ be denoted by $(x_1^i, x_2^i)$. Then, the coupling input is chosen to enter only $x_1^i$ dynamics and is selected as $\epsilon \sum_{j \in \mc N_i} (x_1^j  - x_1^i)$, where $\epsilon=0.1$ is the coupling strength and $\mc N_i$ is the set of neighbors of oscillator $i$. The simulation results are shown in Figure~\ref{fig:vanderpol-coupling}. It is easy to verify that since the underlying interaction graph is connected, the conditions of Proposition~\ref{prop:synch-deterministic-coupling} are satisfied and synchronization is achieved. 
\end{example}

\vskip .2in

\textbf{Case 3. Stochastic coupling and stochastic averaging theory.}  Finally, we consider $N$ oscillators which are connected through a common noise and  deterministic and stochastic coupling among themselves, i.e., consider \eqref{eq:full-couplde-SDE} with nonzero coupling $\H$ and nonzero stochastic coupling $\C$:
\begin{align}\label{eq:couplde-SDE-with-delta-m-dim}
d \theta_i & =  \left(\omega+\epsilon\sum_j c_{ij} \bs Z^\top \H +
  \frac{\sigma^2 }{2}  \mathrm{tr}[\B^\top \bs H \B]+  \frac{\delta^2 }{2} \mathrm{tr}\Big[\sum_j c_{ij}^2 \C^\top \bs H \C\Big] \right)\;dt \nonumber
\\
 &\quad+  \sigma  \bs Z^\top \B  \;d  W_i(t) + \delta\sum_j c_{ij} \bs Z^\top\C \;d  W_{ij}(t). 
\end{align}
Similar to Case 1, to apply Theorems \ref{thm:cond:synch2}-\ref{thm:cond:synch4}, \eqref{eq:couplde-SDE-with-delta} must be of the format of \eqref{eq:general-form2}, with one dimensional $dW_i$ and $dW_{ij}$. Assume identical $dW_i$s and let $\ino$ and $dW$ be as defined in Case 1. Also, let   
\[\hat \nc (\theta_i,\theta_j) = \sqrt{\bs Z(\theta_i)^\top (\C\C^\top)(\gamma(\theta_j), \gamma(\theta_i))\bs Z(\theta_i)},\]
 and write \eqref{eq:couplde-SDE-with-delta-m-dim} as
\begin{align}\label{eq:couplde-SDE-with-delta}
d \theta_i & =  \left(\omega+\epsilon\sum_j c_{ij} \bs Z^\top \H +
  \frac{\sigma^2 }{2}  \mathrm{tr}[\B^\top \bs H \B]+  \frac{\delta^2 }{2} \mathrm{tr}\Big[\sum_j c_{ij}^2 \C^\top \bs H \C\Big] \right)\;dt 
+  \sigma  \ino  \;d  W(t) + \delta\sum_j c_{ij} \hat\nc \;d  W_{ij}(t), 
\end{align}
where $dW$ and $dW_{ij}$ (with a slight abuse of notation) are one dimensional Wiener increments and $\ino$ and $\hat\nc$ are scalars. 

For any $i=1,\ldots, N$, let 
\begin{align}\label{eq:psi}
\psi_i (t)=
 &\theta_i(t) -
  \omega t - 
  \frac{\sigma^2 }{2} \int_0^t \mathrm{tr}[\B^\top \bs H \B](\theta_i(\tau)) \;d\tau -
  \sigma \int_0^t \ino(\theta_i(\tau)) \;d  W(\tau). 
\end{align}
Then, for any $i$,
\begin{align*}
 d \psi_i(t) & =  \epsilon\sum_j c_{ij}\bs Z(\theta_i(t))^\top \H(\gamma(\theta_j(t)), \gamma(\theta_i(t)))\;dt+
 \frac{\delta^2 }{2}  \mathrm{tr}\Big[\sum_j c_{ij}^2 \C(\gamma(\theta_j), \gamma(\theta_i))^\top \bs H(\theta_i(\tau)) \C(\gamma(\theta_j), \gamma(\theta_i))\Big]\;dt \nonumber\\
 &\quad+
  \delta\sum_j c_{ij} \hat\nc(\theta_j(t), \theta_i(t)) \;d  W_{ij}(t). 
 \end{align*}
Note that  $d \psi_i(t)$ has two deterministic terms of order $\epsilon$ and one stochastic term of order $\sqrt\epsilon$, since in \eqref{order-coupling-noise}, we assumed that $\mathcal{O}(\delta)= \mathcal{O}(\sqrt\epsilon)$. 
Therefore, to keep $d \psi_i(t)$ of order $\epsilon$ in the deterministic terms and of order $\mathcal{O}(\delta)$ in the stochastic term,  we approximate the arguments $\theta_i(t)$
and $\theta_j(t)$ in the right hand side of $d \psi_i(t)$ by $\omega t + \psi_i$ and $\omega t + \psi_j$, respectively, and ignore the terms of order $\sigma$ and $\sigma^2$. 
\begin{align}\label{eq:d-psi:before:averaging}
 d \psi_i(t) & =  \epsilon\sum_j c_{ij}\bs Z(\omega t+\psi_i(t))^\top \H(\gamma(\omega t+\psi_j(t), \gamma(\omega t+\psi_i(t)))\;dt\nonumber \\
  &\quad+ \frac{\delta^2 }{2}  \mathrm{tr}\Big[\sum_j c_{ij}^2 \C(\gamma(\omega t+\psi_j(t)), \gamma(\omega t+\psi_i(t)))^\top \bs H(\omega t+\psi_i(t)(\tau)) \C(\gamma(\omega t+\psi_j(t)), \gamma(\omega t+\psi_i(t)))\Big] \;dt\nonumber \\
 &\quad+ \delta\sum_j c_{ij} \hat\nc(\omega t+\psi_j(t), \omega t+\psi_i(t)) \;d  W_{ij}(t). 
 \end{align}
 
 Similar to the argument that we had below \eqref{eq:couplde-SDE-no-delta}, here we use an averaging theory to approximate  \eqref{eq:d-psi:before:averaging} so that the conditions of Theorems \ref{thm:cond:synch2}-\ref{thm:cond:synch4} hold. Since \eqref{eq:d-psi:before:averaging} is an SDE, we employ a stochastic version of averaging theory, as described below. 
 
In the following proposition, which is slightly modified from the materials in 
\cite[Chapter 7, Section 9]{Random_perturbations_DS_book} \footnote{The non-autonomous SDE \eqref{SDE:AveragingPrinciple} is equivalent to
 \cite[Equation (9.1)]{Random_perturbations_DS_book} 
  when $l=1, B=1, C=0$, and $y=0$. Also, \eqref{alpha-bar} and \eqref{beta-bar} are respectively equivalent to  \cite[Equations (9.3) and (9.10)]{Random_perturbations_DS_book}.},
 we state the Averaging Theory for SDEs,  analogs to the Averaging Theory for ODEs.  

\begin{proposition}[\bit{An averaging principle for SDEs}]\label{Prop:AveragingPrinciple}
Consider the time dependent stochastic differential equation
 \begin{equation}\label{SDE:AveragingPrinciple}
 d X = \varepsilon \alpha(X,t) dt + \sqrt \varepsilon \beta (X,t) dW^{(n)}(t),
 \end{equation} 
 where $X\in \real^r, {\alpha: \real^r\times(-\infty,\infty)\mapsto \real^r, \beta: \real^r\times(-\infty,\infty)\mapsto \real^{r\times n}}$, $W^{(n)}$ is an $n-$dimensional Wiener process, and $\varepsilon$ is a  small time scale,  $0<\varepsilon \ll1$. 
Assume that the entries of $\alpha(X,t)$ and the diffusion matrix $\beta (X,t)$ are $T-$periodic functions on $t$ and let
 \begin{align}
\bar \alpha (X) &= \dfrac{1}{T} \int_0^T \alpha(X,s) \; ds, & {\bar\alpha: \real^r\mapsto\real^r}, \label{alpha-bar}\\
(\bar \beta (X))_{ij}&= \left(\dfrac{1}{T} \int_0^T \left(\beta(X,s) \beta^\top(X,s)\right)_{ij}  \; ds \right)^{\frac{1}{2}}, &{\bar\beta: \real^r\mapsto \real^{r\times r}}\label{beta-bar}, 
 \end{align} 
 be the averages  of $\alpha$ and $\beta$, respectively, where $(A)_{ij}$ denotes the $(i,j)^{th}$ entries of a matrix $A$.  Then, for a sufficiently small $\epsilon>0$,  any solution of the non-autonomous equation \eqref{SDE:AveragingPrinciple}, denoted by $X(t)$, can be approximated by $\bar X(t)$, a solution of the following autonomous equation
  \begin{equation}\label{AveSDE:AveragingPrinciple}
 d \bar X = \varepsilon \bar \alpha(X) dt +\sqrt\varepsilon \bar \beta (X) dW^{(r)}(t), 
 \end{equation} 
 where $W^{(r)}$ is an $r-$dimensional Wiener process. 
\end{proposition}

We apply Proposition \ref{Prop:AveragingPrinciple}, with $r=1$,  to the scalar SDE \eqref{eq:d-psi:before:averaging} and obtain:
 \begin{align}\label{eq:d-psi:after:averaging}
d \psi_i \;=\; \epsilon \displaystyle\sum_{j=1}^N c_{ij}\mathcal{H}_1(\psi_j-\psi_i) \; dt
+\delta^2\displaystyle\sum_{j=1}^N c_{ij}^2\mathcal{H}_2(\psi_j-\psi_i) \; dt
+\delta  \displaystyle\sum_{j=1}^N c_{ij}\nc (\psi_j-\psi_i) d  W_{ij}(t), 
   \end{align}
where the drift and the diffusion terms are
\begin{align}
\mathcal{H}_1(\phi) &=  \dfrac{1}{{2\pi}}  \int_0^{2\pi} \bs Z(\xi)^\top \H(\gamma(\xi+\phi), \gamma(\xi))\;d\xi, \label{phase:coef:H1}
\\
\mathcal{H}_2(\phi) &= \dfrac{1}{4\pi}   \int_0^{2\pi} \mathrm{tr}\big[\C(\gamma(\xi+\phi), \gamma(\xi))^\top  \bs H(\xi) \C(\gamma(\xi+\phi), \gamma(\xi))\big] \; d\xi\label{phase:coef:H2},
\\
\nc^2 (\phi) &=  \dfrac{1}{2\pi}   \int_0^{2\pi}  \hat\nc^2(\xi+\phi, \xi) \;d\xi
                    = \dfrac{1}{2\pi}   \int_0^{2\pi}  \bs Z(\xi)^\top (\C\C^\top)(\gamma(\xi+\phi), \gamma(\xi))\bs Z(\xi)\;d\xi
\label{phase:coef:K2}.
\end{align}

Combining \eqref{eq:couplde-SDE-with-delta}, \eqref{eq:psi}, and \eqref{eq:d-psi:after:averaging}, we get 
\begin{align}\label{eq:couplde-SDE-with-delta-averaging}
d \theta_i & =  \left(\omega+
  \frac{\sigma^2 }{2}  \mathrm{tr}[\B^\top \bs H \B]+
   {\delta^2 } \sum_j c_{ij}^2 \mathcal{H}_2(\theta_j-\theta_i) +
  \epsilon\sum_j c_{ij} \mathcal{H}_1(\theta_j-\theta_i)
   \right)\;dt \nonumber
\\
 &\quad+  \sigma \ino(\theta_i) \;d  W(t) + \delta\sum_j c_{ij} \nc(\theta_j-\theta_i) \;d  W_{ij}(t). 
\end{align}
Note that we used 
$ \mathcal H_i(\psi_j-\psi_i) \approx  \mathcal H_i(\theta_j-\theta_i)$, since we are only interested in deterministic terms of order $\epsilon$ and used $ \nc(\psi_j-\psi_i) \approx  \nc(\theta_j-\theta_i)$, since we are only interested in stochastic terms of order $\sqrt\epsilon$ (or $\delta$). 

Now we check the conditions of Theorems \ref{thm:cond:synch2}-\ref{thm:cond:synch3}, where 
$\mathcal F = \omega +
  \frac{\sigma^2 }{2}  \mathrm{tr}[\B^\top \bs H \B]$ 
and $\nc$ is as defined in \eqref{phase:coef:K2}. 
 Here, we assume that $c_{ij}=1$, then  $\mathcal{H} = \mathcal{H}_1+\mathcal{H}_2$ where $\mathcal{H}_i$s are as defined in \eqref{phase:coef:H1}-\eqref{phase:coef:H2}. 
 
\begin{description}[leftmargin=*]
\item[Condition i of Theorem \ref{thm:cond:synch2}.] 
Assume that $K$ is differentiable and let 
\begin{equation}\label{barcF}
\hat c_{\mathcal{F}} = {\sup_{\phi \in (-\frac{\pi}{2}, \frac{\pi}{2})}} \frac{N}{2}   \lambda_{\max} \left[\frac{\partial }{\partial \phi}  \left(\B(\gamma(\phi))^\top  \bs H(\phi) \B(\gamma(\phi))\right)
\right], 
\end{equation}
where $\lambda_{\max}[A]$ denotes the maximum eigenvalue of $A$.  Then, for $\phi, \eta\in {(-\frac{\pi}{2}, \frac{\pi}{2})}$
\begin{equation*}
(\phi-\eta) (\mathcal{F}(\phi)-  \mathcal{F}(\eta))
\leq  \bar c_{\mathcal{F}}(\phi-\eta)^2, \quad  \bar c_{\mathcal{F}} = \sigma^2 \hat c_{\mathcal{F}}. 
\end{equation*}
We used $\mathrm{tr}[A] = \sum_{i=1}^N e_i^\top A e_i$,  where $e_i$s are the standard basis of $\mathbb R^N$, and 
$e_i^\top A e_i \leq \lambda_{\max}[A] e_i^\top e_i.$

\item[Condition ii of Theorem \ref{thm:cond:synch2}.] 
Assume that $H$ and $\C$ are appropriate functions such that $\mathcal H=\mathcal H_1 +\mathcal H_2$ is an odd function. Then $\phi\mathcal{H} (\phi) \geq  \underbar c_{\mathcal H}\phi^2,$ where 
\begin{equation}\label{underbarcH}
 \underbar c_{\mathcal{H}}
 = \inf_{\phi\neq0{\in(-\frac{\pi}{2}, \frac{\pi}{2})}}
 \frac{\phi\mathcal{H}_1(\phi)+\phi\mathcal{H}_2(\phi)}{\phi^2},
\end{equation}

\item[Condition iii of Theorem \ref{thm:cond:synch2}.] 
$|\nc(\phi)| \leq \bar c_{\nc} |\phi|$, where
\begin{equation}\label{barC}
 \bar c_{\nc}:= \sup_{\phi\neq0\in(-\frac{\pi}{2}, \frac{\pi}{2})}\dfrac{|\nc(\phi)|}{|\phi|}. 
\end{equation}
 
\item[Condition i of Theorem \ref{thm:cond:synch3}.] If we can choose an appropriate  $\C$ that guarantees  $\mathcal H$ is odd and $\nc$ is even odd or even, then $ \underbar c_{\nc} |\phi|\leq|\nc(\phi)|$, where 
\begin{equation}\label{underbarC}
 \underbar c_{\nc}:= \inf_{\phi\neq0\in(-\frac{\pi}{2}, \frac{\pi}{2})}\dfrac{|\nc(\phi)|}{|\phi|}. 
\end{equation}

\item[Condition iv of Theorem \ref{thm:cond:synch2} and condition ii of Theorem \ref{thm:cond:synch3}.] 
Assume that $\B$ is differentiable. Then 
$\ino(\phi,t) = \sqrt{\bs Z(\omega t+\phi)^\top (\B\B^\top)(\gamma(\omega t+\phi))\bs Z(\omega t+\phi)}$
becomes differentiable and  for any $\phi,\eta \in (-\frac{\pi}{2}, \frac{\pi}{2})$ and $t\in[0,T)$, 
\begin{equation}\label{barcK}
|\ino(\phi,t)-\ino(\eta,t)| \leq   \bar c_{\ino}  |\phi-\eta|,
 \quad \text{where}\quad
\bar c_{\ino} :=  
{\sup_{(\phi,t) \in (-\frac{\pi}{2}, \frac{\pi}{2})\times[0,T)}}
\frac{\partial }{\partial \phi} \ino(\phi,t),
\end{equation}
and 
\begin{equation}\label{underbarcK}
 \underbar c_{\ino}  (\phi-\eta)^2 \leq (\phi-\eta) (\ino(\phi,t)-\ino(\eta,t)),
 \quad \text{where}\quad
 \underbar c_{\ino} :=  {\inf_{(\phi,t) \in (-\frac{\pi}{2}, \frac{\pi}{2})\times[0,T)}}
\frac{\partial }{\partial \phi}  \ino(\phi,t). 
\end{equation}

\end{description}

In summary, we proved the following proposition. 

\begin{proposition}
Consider  \eqref{eq:couplde-SDE-with-delta} with $c_{ij}=c_{ji}=1$. Assume $\B$ is differentiable and $\C$ and $H$ are chosen such that $\mathcal{H}$ is an odd function. Then \eqref{eq:couplde-SDE-with-delta} stochastically synchronizes if one of the following holds:
\begin{enumerate}
\item $c: =-2\sigma^2 \hat  c_{\mathcal{F}} +  {2}\epsilon \underbar c_{\mathcal{H}}  \lambda-\delta^2  \bar c_{\nc}^2 (1-\frac{1}{N}) \lambda_{N} - \sigma^2 \bar c_{\ino}^2>0.$
\item $c<0$  but $c+ 2 \sigma^2 \underbar c_{\ino}^2>0.$
\end{enumerate}
If  $\C$ can make $\nc$ an odd or an even function, then \eqref{eq:couplde-SDE-with-delta} stochastically synchronizes if $c>0$ or 
$c<0$  but $c+ 2 \sigma^2 \underbar c_{\ino}^2 + \dfrac{\delta^2 \underbar c_{\nc}^2 \lambda_2^2}{6} >0.$
The constant bounds are defined in \eqref{barcF}-\eqref{underbarcK}. 
\end{proposition}

\section{Conclusions}\label{sec:conclusions}

In this paper, we first considered networks of non-linear systems with non-linear coupling which are driven by two sources of state-dependent white noise: a common noise and a noise generated by the interactions between the systems. 
We provided sufficient conditions that guarantee stochastic synchronization 
in such noisy networks and discussed the cases that noise can be useful or harmful for network synchronization. Next, we focused on networks of oscillators (instead of any arbitrary system) which are weakly coupled (instead of any arbitrary coupling) and using the notion of first- and second-order PRCs and averaging theory for deterministic and stochastic systems, we derived the corresponding stochastic phase equations. Then, we examined the synchronization conditions that we found in the first part of the paper, and provided new synchronization conditions in terms of the first- and second-order PRCs. We provided numerical examples to illustrate our results. 
 
 Future directions for investigation include:
(1) Leveraging the one dimensional  phase reduction of an $m$ dimensional noisy oscillator, we will compute various statistical properties of the time period of the noisy limit cycle. One period of the noisy limit cycle can be interpreted as the first passage/hitting time of the phase variable $\theta$  evolving according to \eqref{eq:full-couplde-SDE} with absorbing boundary at $\theta= 2\pi$ and initial condition at $\theta= 0$. 
(2) Using the corresponding Fokker-Planck equations, we will analytically estimate the moments of the time period of the synchronization solution that we discussed in this paper. Comparing the moments of the time periods of the synchronization solution and each isolated oscillator would allow us to understand the effect of coupling functions on the precision of the oscillators.
(3) We will consider a network of heterogeneous systems and provide conditions for stochastic cluster synchronization.

\section*{Acknowledgement}
This work was supported in part by  Simon Foundations grant 712522 and ARO grant W911NF-18-1-0325.

\section*{Appendix}
In this appendix we provide a few examples to clarify the difference between our approach and the existing approaches.  

\begin{example}
Consider the following SDE with a constant diffusion term, i.e., the It\^{o} and Stratonovich interpretations are identical:
\begin{equation}\label{app:sde1}
dx= ax dt+ b \; dW.
\end{equation}
Now we let $z=\sin(x)$ and compute $dz$. There are two ways to compute $dz$. 

First, we interpret  \eqref{app:sde1} in  Stratonovich and apply the ordinary chain rule. This gives the following SDE in  Stratonovich:
\[dz= ax \cos(x) dt+ b\cos(x) dW,\]
which by It\^{o} lemma, it transfers to the following SDE in It\^{o}:
\begin{equation}\label{app:sde1:Str:Ito}
dz= \left(ax \cos(x) - \frac{b^2}{2}\sin(x)\cos(x)\right) dt+ b\cos(x) dW.
\end{equation}

Second, we interpret  \eqref{app:sde1} in  It\^{o} and apply the It\^{o}  chain rule. This gives the following SDE in  It\^{o}:
\begin{equation}\label{app:sde1:Ito}
dz= \left(ax \cos(x) - \frac{b^2}{2}\sin(x)\right) dt+ b\cos(x) dW.
\end{equation}

Although both  \eqref{app:sde1:Str:Ito} and  \eqref{app:sde1:Ito} describe  $dz$ in the It\^{o} sense, they are not identical. 

\end{example}

\begin{example}
Consider the following Ornstein–Uhlenbeck SED with Stratonovich interpretation:
\begin{equation}\label{app:sde2}
dx= ax dt+ bx  dW.
\end{equation}
Now we let $z=\ln(x)$ and compute $dz$. Similar to the previous example, we employ two ways to compute $dz$. 

First, since \eqref{app:sde2} is interpreted as Stratonovich, we apply the ordinary chain rule. This gives the following SDE in  Stratonovich:
\[dz= a dt+ b \; dW.\]
Now by It\^{o} lemma, we transfer it to the following SDE in the It\^{o} sense:
\begin{equation}\label{app:sde2:Str:Ito}
dz= a dt+ b \; dW.
\end{equation}

Second, we transfer  \eqref{app:sde2} to an SDE with  It\^{o} interpretation.  This gives the following SDE in  It\^{o}:
\[ dx = \left(ax-\frac{b^2x}{2} \right) dt+b  \;dW.\]
Now, we apply the It\^{o}  chain rule to get:
\begin{equation}\label{app:sde2:Ito}
dz= \left(a-\frac{b^2}{2} \right) dt+b \; dW. 
\end{equation}

Note that both  \eqref{app:sde2:Str:Ito} and  \eqref{app:sde2:Ito} describe  $dz$ in the It\^{o} sense, however, they are not identical. 
\end{example}

\begin{example}
Now we consider the Van der Pol oscillator discussed in Example \ref{example:vanderpol1} with $\sigma\neq0$ and compute $d\theta$ where $\theta$ is the corresponding phase. Following the first method discussed in the above examples, $d\theta$ becomes:
\begin{equation}\label{app:sde3:SDE:Ito}
d\theta= \left(\omega + \frac{\sigma^2}{2} \bs{Z}(\theta)^\top \bs{Z}'(\theta) \right) dt+\sigma \bs{Z}(\theta)^\top dW, 
\end{equation}
and following the second method discussed in the above examples, $d\theta$ becomes:
\begin{equation}\label{app:sde3:Ito}
d\theta= \left(\omega+ \frac{\sigma^2}{2} \mathrm{tr}[\bs H(\theta)]\right) dt+\sigma \bs{Z}(\theta)^\top dW. 
\end{equation}
Fig~\ref{fig:vanderpol_appendix} shows that $\bs{Z}(\theta)^\top \bs{Z}'(\theta)$ and $\mathrm{tr}[\bs H(\theta)]$ are not identical. Therefore,  \eqref{app:sde3:SDE:Ito} which is derived in \cite{2011_Ermentrout_Beverlin_Troyer_Netoff, teramae2004robustness, teramae2009stochastic} and  \eqref{app:sde3:Ito} which is derived in this work, are not identical. 
 \begin{figure}[ht!]
        \centering
        \includegraphics[width=0.4\linewidth]{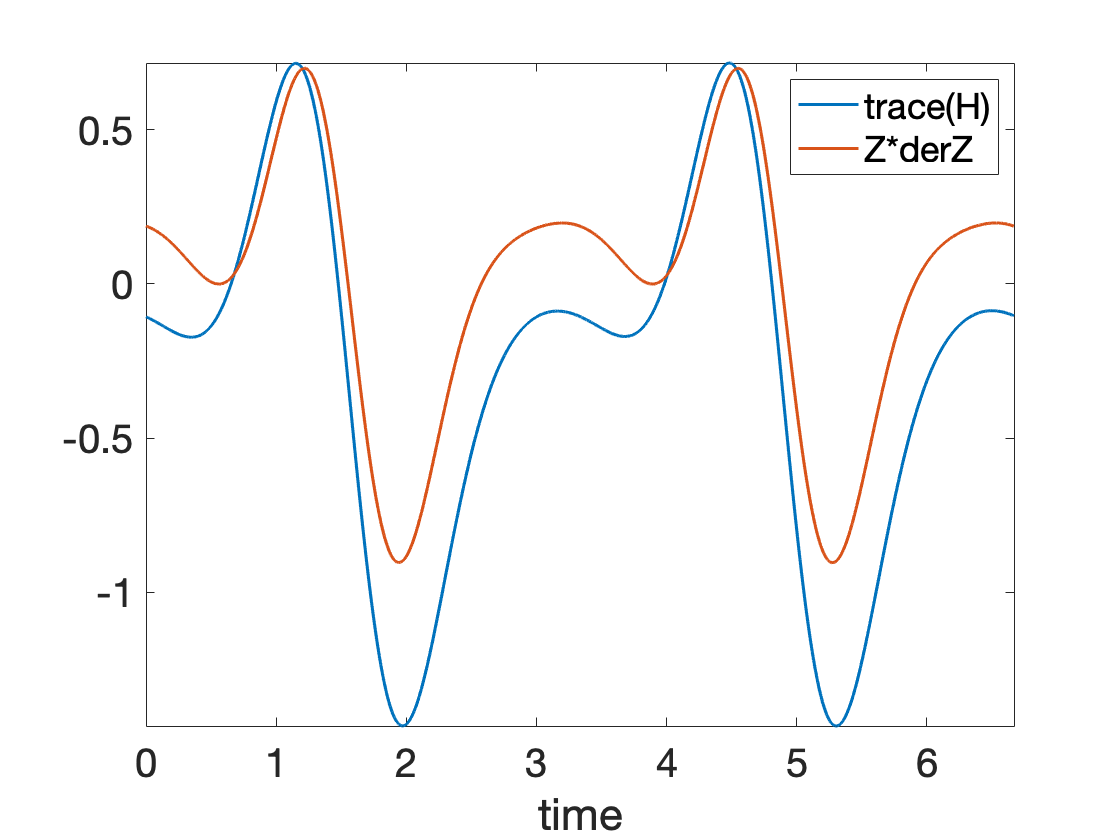}
        \caption{The comparison of $\bs{Z}(\theta)^\top \bs{Z}'(\theta)$ and $\mathrm{tr}[\bs H(\theta)]$ in the  Van der Pol oscillator.}
        \label{fig:vanderpol_appendix}
    \end{figure}
\end{example}


\end{document}